\documentclass{amsart}
\usepackage{amsmath}
\usepackage{mathrsfs}
\usepackage{amssymb}
\usepackage{enumerate}
\usepackage{bbm}
\usepackage{hyperref}

\newtheorem{thm}{Theorem}[section]
\newtheorem{cor}[thm]{Corollary}
\newtheorem{prop}[thm]{Proposition}
\newtheorem{lem}[thm]{Lemma}
\newtheorem{rem}[thm]{Remark}
\newtheorem{defn}[thm]{Definition}

\newtheorem{exa}[thm]{Example}

\DeclareMathOperator{\diam}{diam}

\numberwithin{equation}{section}

\numberwithin{equation}{section}
\newcommand{\N}{\mathbb{N}}
\newcommand{\K}{\mathscr{K}}
\newcommand{\Z}{\mathbb{Z}}
\newcommand{\Zp}{\mathbb{N}_0}
\newcommand{\R}{\mathbb{R}}
\newcommand{\set}[1]{\{#1\}}
\newcommand{\eps}{\varepsilon}
\newcommand{\F}{\mathcal{F}}
\newcommand{\M}{\mathscr{M}}

\begin{document}

\title[On average error in tracing]{On various definitions of shadowing with average error in tracing}

\author{Xinxing Wu}
\address{School of Mathematics, University of Electronic
Science and Technology of China, Chengdu, Sichuan, 611731, People's
Republic of China} \email{wuxinxing5201314@163.com}
\author{Piotr Oprocha${}^*$}
\address{Faculty of Applied Mathematics, AGH University of Science and
Technology, al. A. Mickiewicza 30, 30-059, Krak\'{o}w, Poland}
\email{oprocha@agh.edu.pl}
\author{Guanrong Chen}
\address{Centre for Chaos and Complex Networks, City University of Hong Kong, Hong Kong SAR, People's
Republic of China} \email{eegchen@cityu.edu.hk}
\thanks{${}^*$Piotr Oprocha (Corresponding Author).}


\maketitle
\centerline{(Communicated by the associate editor name)}
\begin{abstract}
In this paper we present a systematic study of shadowing properties with average error in tracing such as
(asymptotic) average shadowing, $\underline{d}$-shadowing, $\overline{d}$-shadowing and almost specification.
As the main tools we provide a few equivalent characterizations of the average shadowing property,
which also partly apply to other notions of shadowing. We prove that almost specification on the whole space induces this property on the measure center.
Next, we show that always (e.g. without assumption that the map is onto) almost specification implies asymptotic average shadowing,
which in turn implies the average shadowing property and consequently also $\underline{d}$-shadowing and $\overline{d}$-shadowing. Finally, we study connections among sensitivity, transitivity, equicontinuity and (average) shadowing.
\end{abstract}

\section{Introduction}

The theory of shadowing provides tools for fitting real trajectories nearby to approximate trajectories.
The motivation comes from computer simulations, where we always have a numerical error when calculating a trajectory, no matter how small,
but at the same time we want to be sure that what we see on the computer screen is a good approximation of the genuine orbit of the system.
It is a classical notion, which originated from works of Anosov, Bowen and others (see \cite{PilBook} for historical remarks and more recent advances).
Lately, some equivalent conditions for expansive homeomorphisms having the shadowing
property are obtained, e.g. see \cite{KO12, LS2005, PilBook, Sakai2001, Sakai2003}. For example, Lee and Sakai compared various shadowing properties for
(positively) expansive maps in \cite{LS2005, Sakai2001, Sakai2003} and proved that the continuous shadowing property, the
Lipschitz shadowing property, the limit shadowing property and the strong shadowing
property are all equivalent to the (usual) shadowing property for expansive homeomorphisms on compact metric spaces (see also \cite{KO12}).

But it is also easy to imagine situations where we cannot provide an
exact bound for the error in each step of computing the orbit; however, we can ensure that
the average error in the long run is small. This was the motivation of Blank, who introduced
the notion of the average shadowing property (see \cite{Blank1988, Blank1988-1}) and
proved that $f|_{\Lambda}$ has the average shadowing property provided that
$\Lambda$ is a basic set of a diffeomorphism $f$ satisfying Axiom A.
First examples of dynamical systems with the average shadowing property were obtained on
manifolds in the class of diffeomorphisms \cite{Blank1988, Sakai} (and their random perturbations), although this property is much more common.
By results recently published in \cite{KO2010, KO11}, the average shadowing property is present in all mixing interval maps, mixing sofic shifts and all $\beta$-shifts (symbolic descriptions of $\beta$-transformations). In fact, all these systems have the so-called almost specification property, which is a generalization of Bowen's specification property in terms of average tracing. This concept was introduced by Pfister and Sullivan \cite{PS}, and then extended and studied by Climenhaga and Thompson \cite{CT2013} and many others (e.g. see also \cite{CT2012,T,Y09} or \cite{OT}).
The asymptotic average shadowing property was introduced
by Gu \cite{Gu2007-1}, which followed the same framework as Blank in \cite{Blank1988},
but with the limit shadowing property in place of the shadowing property as the starting point for generalization.
Meanwhile, he showed that every surjective dynamical system with asymptotic average shadowing
is chain transitive and every $\mathcal{L}$-hyperbolic homeomorphism with asymptotic average shadowing
is topologically transitive. Then, Honary and Bahabadi \cite{HB2008} proved that
for a dynamical system $(X, f)$, if the chain recurrent set $R(f)$ has more than one chain
component, then $f$ does not have the asymptotic average shadowing property and that
$R(f)$ is the single attractor for $f$ provided that $f$ has the asymptotic average shadowing property.

In some sense, parallel approach to average tracing was initiated in \cite{DH2010} (see also \cite{FG2010}),
where the authors introduced the concept of partial shadowing.
This concept was further developed in \cite{ODH2014}. The main idea is to use Furstenberg families to specify the minimal requirements on the set
of iterations, where a proper $\eps$-tracing takes place. In \cite{ODH2014}, it was also observed that some particular cases of this kind of shadowing can be induced by the average shadowing property.

The present work is inspired by the concepts and results from the papers mentioned above
and is organized as follows.
First, we review some standard notions to be used in the paper. In Section~\ref{sec:Ma}, we
analyze basic properties of $\M_\alpha$ and $\M^\alpha$-shadowing properties (see Definition~\ref{def:F-shadowing}),
which are generalizations of $\underline{d}$-shadowing and $\overline{d}$-shadowing considered in \cite{DH2010} and are special types of $\F$-shadowing from \cite{ODH2014}. In this section we prove that both $\M_\alpha$ and $\M^\alpha$-shadowing properties are preserved by higher iterations of the map
(see Theorem~\ref{M_aiterations}),
which also shows that every surjective dynamical system with $\overline{d}$-shadowing is chain mixing (see
Corollary~\ref{Corollary XX}).

We focus on some natural problems form \cite{KKO14}, which shows that the almost specification property implies
asymptotic average shadowing, consequently implies average shadowing, provided
that the dynamical system is surjective (or even chain mixing).
It is also proved in \cite{KKO14} that if a dynamical system has the average shadowing property, asymptotic average shadowing
property or almost specification property, when restricted at the measure center, then it also satisfies the respective property on the whole space. Then \cite[Question 10.1]{KKO14} asks if the converse is true and the authors also mention that if the answer is positive then the assumption
of surjection can be removed when proving implications between these properties.

In the present paper, we provide only a partial answer to \cite[Question 10.1]{KKO14}. However, we
prove that the assumption of surjection is not essential, that is, the first result mentioned above holds also
in the general case. Precisely speaking, in Theorem~\ref{Theorem 5.1} we show that the
asymptotic average shadowing property implies a property which we call weak asymptotic
average shadowing property. Then, in Section~\ref{sec:equivASP}, we show that the weak asymptotic
average shadowing property in fact coincides with the average shadowing property, and that it can be
equivalently characterized in the language of the $\M_\alpha$-shadowing property (see Theorem~\ref{Theorem 6.5}).

Theorem~\ref{Theorem 7.2} indicates that if $A$ is a closed invariant set containing the measure center then the $\mathscr{M}_{\alpha}$-shadowing property
of $(A, f|_{A})$ guarantees the same property for $(X, f)$.
Theorem~\ref{mc:almost_s} shows that the almost specification property of $(X,f)$ implies almost specification of $(A,f|_A)$, answering (partly) \cite[Question 10.1]{KKO14}. In particular, this implies (by results of \cite{KKO14})
that the almost specification property always implies the asymptotic average shadowing property (it does not matter if $f$ is surjective or not).

The paper ends with Section~\ref{sec:sens}, where we prove that $\underline{d}$-shadowing and $\overline{d}$-shadowing  properties
are strongly related with transitivity and sensitivity. In particular, there are no interesting equicontinuous examples of systems with these properties.
It is also worth recalling at this point that the usual version of the shadowing property is present in all equicontinuous systems on any totally disconnected space.

\section{Preliminaries}
The set of real numbers, integers, natural numbers and nonnegative
integers are denoted, respectively, by $\R$, $\Z$, $\N=\Z \cap
(0,+\infty)$ and $\Zp=\N \cup \set{0}$. If $A$ is a set, then its
complement is denoted $A^c$ and its closure $\overline{A}$.
Cardinality of a set $A$ is denoted $|A|$.
If $A\subset \Zp$ and $j\in \Zp$ then the standard notation will be used
\begin{eqnarray*}
A+j&=&\left\{i+j: i\in A\right\}, \\
A-j&=&\left\{i-j: i\in
A\right\}\cap \Zp,\\
j\cdot A& =& \left\{j\cdot i: i\in A\right\}.
\end{eqnarray*}
Clearly, for any choice of sets $A_{1},\ldots, A_n\subset \Zp$, any $n\in \N$ and any $k\in \Zp$,
\begin{equation}\label{*}
\bigcap_{i=1}^{n}(A_{i}+k)=\left(\bigcap_{i=1}^{n}A_{i}\right)+k.
\end{equation}
and
\begin{equation}\label{**}
\begin{split}
|A_1\cap A_2 \cap\{0, \ldots, n-1\}|&=|A_1\cap\{0, \ldots, n-1\}|
+|A_{2} \cap\{0, \ldots, n-1\}|\\
&-|\left(A_1\cup A_2\right) \cap\{0, \ldots, n-1\}|.
\end{split}
\end{equation}

A set $A\subset\Zp$ is \emph{syndetic} if it has bounded gaps, i.e.,
if there is $k>0$ such that $A\cap [i,i+k]\neq \emptyset$ for all $i\in \Zp$.
The family of all syndetic sets is denoted $\F_s$.

\subsection{Topological dynamics}
A \emph{dynamical system} is a pair $(X, f)$ consisting of a compact metric space
$(X, d)$ and a continuous map $f\colon X\longrightarrow X$.
A dynamical system $(Y, g)$ is a {\it factor} of $(X, f)$ if
there is a continuous surjection $\pi: X\longrightarrow Y$ such that $\pi\circ f= g\circ \pi$.

A set $M\subset X$ is \emph{minimal} if it does not have closed, nonempty and invariant subsets other than itself. The set of minimal points (elements of minimal sets) in $(X,f)$ is denoted $M(f)$.


For any open sets $U, V$, define \emph{the set of transfer times} by
$N_f(U, V)=\set{n\in \Zp : f^n(U) \cap V\neq \emptyset}$.
Similarly, for any $x\in X$, let $N_f(x,V)=\set{n\in \Zp : f^n(x) \cap V\neq \emptyset}$.
When the map $f$ is clear from the context,
we simply write $N(U, V)$ and $N(x, V)$.

The \emph{orbit} of $x\in X$ is the set $\left\{f^{n}(x): n\in \Zp\right\}$. We say that
$f$ is \emph{minimal} if the orbit of every point $x\in X$ is dense in $X$. It is easy
to see that $f$ is minimal if and only if $X$ has no proper, nonempty, closed invariant subset.
The map $f$ is \emph{transitive} (resp. \emph{syndetically transitive})
if, for any pair of nonempty open subsets  $U, V\subset X$, $N_{f}(U, V)\neq \emptyset$
(resp., $N_{f}(U, V)\in \F_{s}$). We say that $f$ is \emph{totally transitive}
if $f^{n}$ is transitive for all $n\in \mathbb{N}$. The map $f$
is \emph{weakly mixing} if $f\times f$ is transitive on $X\times X$.

The map $f$ is \emph{equicontinuous} if, for any $\eps>0$, there exists a $\delta>0$
such that for all $x, y\in X$ with $d(x, y)<\delta$ and all $n\in \Zp$, $d(f^{n}(x),
f^{n}(y))<\eps$.

For $U\subset X$ and $\delta>0$, denote
$$
N_{f}(U, \delta)=\left\{n\in \Zp: \text{there exist } y, z\in U \text{ such that } d(f^{n}(y),
f^{n}(z))>\delta \right\}.
$$

The map $f$ is \emph{sensitive} (resp. \emph{syndetically sensitive}) if there
exists $\delta>0$ such that for every nonempty open subset $U\subset X$, $N_{f}(U, \delta)\neq \emptyset$ (resp. $N_{f}(U, \delta)\in \F_{s}$).

Let $M(X)$ denote the space of all Borel probability measures on $X$.
A measure $\mu\in M(X)$ is {\it invariant} for $f: X\longrightarrow X$ if
$\mu(A) =\mu(f^{-1}(A))$ for any Borel set $A\subset X$.
The classical Krylov-Bogolyubov theorem implies that every compact
dynamical system $(X, f)$ has at least one such measure.

A subset $A\subset X$ is {\it measure saturated} if, for every open set $U$ satisfying
$U\cap A\neq \emptyset$, there exists an invariant measure $\mu$ such that $\mu(U)>0$.
The {\it measure center of $f$} is the largest measure saturated subset.

\subsection{Average tracing of approximate trajectories}

Let $\set{x_i}_{i=0}^\infty$, $\set{y_i}_{i=0}^\infty\subset X$
and fix any $\eps>0$.
We define
\begin{eqnarray*}
\Lambda(\set{x_i}_{i=0}^\infty, \set{y_i}_{i=0}^\infty,f ,
\eps)&:=&\left\{i\in \Zp : d(x_{i}, y_{i})<\eps
\right\},\\
\Lambda^c(\set{x_i}_{i=0}^\infty, \set{y_i}_{i=0}^\infty, f ,\eps)&:=&\Zp \setminus
\Lambda(\set{x_i}_{i=0}^\infty, \set{y_i}_{i=0}^\infty,f ,\eps)
\\
&=& \set{i\in \Zp : d(x_{i}, y_{i})\geq \eps}.
\end{eqnarray*}
When the map $f$ is clear from the context, we simply write
$\Lambda(\set{x_i}_{i=0}^\infty, \set{y_i}_{i=0}^\infty,\eps)$ and $\Lambda^c(\set{x_i}_{i=0}^\infty, \set{y_i}_{i=0}^\infty,\eps)$.
Similarly, we use the following simplified notation (for both
$\Lambda$ and $\Lambda^c$):
\begin{eqnarray*}
\Lambda(\set{x_i}_{i=0}^\infty, f ,\eps)&:=&\Lambda(\set{x_{i+1}}_{i=0}^\infty,
\set{f(x_i)}_{i=0}^\infty, f ,\eps)\\
&=&\set{i\in \Zp : d(f(x_{i}), x_{i+1})<\eps},\\
\Lambda^{c}(\set{x_i}_{i=0}^\infty, f ,\eps)&:=&\Zp
\setminus \Lambda(\set{x_{i+1}}_{i=0}^\infty, f ,\eps)
\\
&=& \set{i\in \Zp : d(f(x_{i}), x_{i+1})\geq\eps},\\
\Lambda(z,\set{x_i}_{i=0}^\infty,f ,\eps)&:=&\Lambda
(\set{f^i(z)}_{i=0}^\infty,\set{x_i}_{i=0}^\infty,f ,\eps)\\
&=& \set{i\in \Zp : d(f^i(z), x_{i})<\eps},\\
\Lambda^c(z,\set{x_i}_{i=0}^\infty,f ,\eps)&:=&\Zp \setminus
\Lambda (z,\set{x_i}_{i=0}^\infty,f ,\eps)\\
&=& \set{i\in \Zp : d(f^i(z), x_{i})\geq \eps}.
\end{eqnarray*}
Finally, we will denote finite blocks in the above sets by
\begin{eqnarray*}
\Lambda_n(\set{x_i}_{i=0}^\infty, \set{y_i}_{i=0}^\infty,f ,\eps)&:=
&[0,n)\cap \Lambda(\set{x_i}_{i=0}^\infty, \set{y_i}_{i=0}^\infty,f ,\eps), \\
\Lambda_n^c(\set{x_i}_{i=0}^\infty, \set{y_i}_{i=0}^\infty,f ,\eps)&:=
&[0,n)\cap \Lambda^c(\set{x_i}_{i=0}^\infty, \set{y_i}_{i=0}^\infty,f ,\eps).\\
\end{eqnarray*}

\begin{defn}
Let $\delta>0$ and let $\xi=\set{x_i}_{i=0}^\infty\subset X$.
We say that $\xi$ is
\begin{enumerate}
\item\label{2.1.1} a \emph{$\delta$-ergodic pseudo-orbit} (of $f$) if
$$
\lim_{n\rightarrow\infty}\frac{1}{n} \left|\Lambda^{c}_{n}(\xi, f,
\delta)\right| =0;
$$

\item\label{2.1.2}
a \emph{$\delta$-average-pseudo-orbit} (of $f$) if there exists $N>0$
such that for all $n\geq N$ and $k\in \Zp$,
$$
\frac{1}{n}\sum_{i=0}^{n-1}d(f(x_{i+k}), x_{i+k+1})<\delta;
$$

\item\label{2.1.3}
a \emph{$\delta$-asymptotic-average-pseudo-orbit} (of $f$) if
$$
\limsup_{n\rightarrow \infty}\frac{1}{n}\sum_{i=0}^{n-1}d(f(x_{i}), x_{i+1})<\delta;
$$

\item\label{2.1.4}
an \emph{asymptotic average pseudo-orbit} (of $f$) if
$$
\lim_{n\rightarrow \infty}\frac{1}{n}\sum_{i=0}^{n-1}d(f(x_{i}), x_{i+1})=0.
$$
\end{enumerate}
\end{defn}

We use the above notions of approximate trajectories to define
three main shadowing properties of the paper.

\begin{defn}
A dynamical system $(X,f)$ has
\begin{enumerate}
\item\label{2.2.1} the \emph{average shadowing property (abbrev. ASP)} if, for
any $\eps>0$ there exists $\delta>0$ such that every
$\delta$-average-pseudo-orbit $\set{x_i}_{i=0}^\infty$ is
$\eps$-shadowed on average by a point $z\in X$, i.e.
$$
\limsup_{n\rightarrow\infty}\frac{1}{n}\sum_{i=0}^{n-1}d(f^{i}(z), x_{i})<\eps;
$$
\item\label{2.2.2} the \emph{asymptotic average shadowing property (abbrev. AASP)}
if every asymptotic average pseudo-orbit $\set{x_i}_{i=0}^\infty$ is
asymptotically shadowed on average by a point $z\in X$, i.e.
$$
\lim_{n\rightarrow\infty}\frac{1}{n}\sum_{i=0}^{n-1}d(f^{i}(z), x_{i})=0;
$$
\item\label{2.2.3} the \emph{weak asymptotic average shadowing
property} if, for any $\varepsilon>0$ and any asymptotic average
pseudo-orbit $\left\{x_{i}\right\}_{i=0}^{\infty}$, there exists
$z\in X$ such that
$$
\limsup_{n\rightarrow\infty}\frac{1}{n}\sum_{i=0}^{n-1}d(f^{i}(z),
x_{i})<\varepsilon.
$$
\end{enumerate}
\end{defn}

A \emph{$\delta$-chain} from $x$ to $y$ is a finite $\delta$-pseudo-orbit between these points, that is,
a sequence $x_1,\ldots, x_{n+1}$ such that $d(f(x_i),x_{i+1})<\delta$ for all $i=1, \ldots, n$, and $x_1=x$, $x_{n+1}=y$.
A map is \emph{chain transitive} if, for any $\delta>0$ and any two points $x,y\in X$ there is a $\delta$-chain from $x$ to $y$.
Chain transitivity is a natural generalization of transitivity. It is clear that if a map is chain transitive
then it must be surjective as well. There is a surprising result \cite{RW2008} which shows that chains do not distinguish between totall
transitivity and mixing. Precisely  speaking, if $(X,f^n)$ is chain transitive for all $n>0$ then it is \emph{chain mixing},
that is, for any $x,y\in X$ and $\delta>0$ there is $N>0$ such that there is a $\delta$-chain from $x$ to $y$ consisting of exactly $n$
elements for every $n>N$.

\subsection{Furstenberg families and tracing}

\emph{A (Furstenberg) family} $\F$ is a collection of subsets of
$\Zp$ which is \emph{upwards hereditary}, that~is
$$
F_1 \in \F \text{ and } F_1 \subset F_2
\quad \Longrightarrow \quad F_2 \in \F.
$$
The {\it dual family} of $\F$ is
$$
\F^*:=\left\{A\subset\Zp :\ \forall \ F\in\F,\ A\cap
F\neq\emptyset\right\}.
$$
A set $A\subset\Zp$ is {\it syndetic} if it has bounded gaps, i.e.
there is $k>0$ such that $A\cap \left[i, i+k\right)\neq \emptyset$
for all $i\geq 0$ and \emph{thick} if it belongs to the dual
family $\F_t=\F^*_s$. Note that a set is thick if it contains
arbitrarily long block of consecutive integers.

For any $A\subset\Zp $, the {\it upper density} of $A$ is defined by

\begin{equation}\label{def:d-}
\overline{d}(A):=\limsup_{n\rightarrow\infty}\frac{1}{n}
\left|A\cap\{0,\,1,\ldots,\,n-1\}\right|.
\end{equation}
Replacing $\limsup$ with $\liminf$ in \eqref{def:d-} gives the
definition of $\underline{d}(A)$, the \emph{lower density} of $A$.
If there exists a number $d(A)$ such that
$\overline{d}(A)=\underline{d}(A)=d(A)$ then we say that the set \emph{$A$ has density $d(A)$}. Fix any $\alpha\in [0,1)$
and denote by $\M_\alpha$ (resp. $\M^\alpha$) the family consisting
of sets $A\subset\Zp$ with $\underline{d}(A)> \alpha$ (resp.
$\overline{d}(A)>\alpha$). We denote by $\hat{\M}_\alpha$ the family
of sets with $\underline{d}(A)\geq \alpha$. Clearly
$\hat{\M}_1$ consists of sets $A$ with $d(A)=1$.
\begin{defn}\label{def:F-shadowing}
A dynamical system $(X,f)$ has \emph{(ergodic) $\F$-shadowing
property} if, for any $\eps>0$ there is $\delta>0$ such that every
$\delta$-ergodic pseudo-orbit $\xi$ is $\F$-$\eps$-shadowed by some
point $z\in X$, i.e.
$$
\Lambda(z,\xi,\eps)\in \F.
$$
In the special case of $\F=\hat{\M}_1$ (resp.,
$\F=\mathscr{M}_{0}$ and $\mathscr{M}^{1/2}$), we say that $(X,f)$ has
the \emph{ergodic shadowing property} (resp.,
\emph{$\underline{d}$-shadowing property} and
\emph{$\overline{d}$-shadowing property}).
\end{defn}

\subsection{The (almost) specification property}

The specification property was first introduced by Bowen \cite{BowenSpec}. It is one of the strongest mixing properties that can be expected from a dynamical system.
A  dynamical system $(X,f)$ has the \emph{strong specification property},
if for any $\eps > 0$ there is a positive integer $M$
such that for any integer $s\geq 2$, any set $\set{y_1,\dots,y_s}$
of $s$ points in $X$, and any sequence $0=j_1\leq k_1 < j_2 \leq k_2
< \dots < j_s \leq k_s$ of $2s$ integers satisfying $j_{m+1} - k_m\geq
M$ for $m= 1,\dots,s-1$, we can find a point $x\in X$ such
that for each positive integer $m\leq s$ and all integers $i$ satisfying
$j_m \leq i \leq k_m$, the following conditions hold:
\begin{eqnarray}
d(f^i(x),f^i(y_m))\!\! &<&\!\! \eps, \label{cond:psp1}\\
 f^n (x)\!\! &=&\!\! x, \;\; \textrm{ where }\; n=M+
k_s.\label{cond:psp2}
\end{eqnarray}
If the only guaranteed condition is \eqref{cond:psp1} (but not necessarily \eqref{cond:psp2}), then
we say that $(X,f)$ has the \emph{specification property}.

Recently, Pfister and Sullivan introduced in \cite{PS} a property called the $g$-almost product property, which generalizes Bowen's specification in terms of average tracing. Inspired by \cite{PS}, Thompson in \cite{T} modified slightly this
definition and proposed to call it the almost specification property,
which in turn generalizes the notion of specification.
In this paper, we adopt the concepts of \cite{T}. First, we introduce some auxiliary notation.

Let $\eps_0 > 0$. A function $g\colon \Zp \times (0, \eps_0] \longrightarrow \N$ is
called a \emph{mistake function} if, for all $\eps\in(0, \eps_0]$ and all $n \in \Zp$,
 we have $g(n,\eps) \leq  g(n+1,\eps)$ and
$$
\lim_{n\to \infty}
\frac{g(n, \eps)}{n}= 0.
$$
Given a mistake function $g$, if $ \eps > \eps_0$, then we define $g(n, \eps) = g(n, \eps_0)$.

For $n$ sufficiently large satisfying $g(n, \eps) < n$, we
define the \emph{set of $(g; n, \eps)$ almost full subsets of $\set{0,\ldots,n-1}$}
as the family $I(g; n, \eps)$ consisting of subsets of $\set{0,1,\ldots,n-1}$
with at least $n - g(n,\eps)$ elements, that is,
$$
I(g; n, \eps) := \left\{A\subset \set{0,1,\ldots, n - 1} :
|A| \geq n - g(n,\eps)\right\}.
$$

For a finite set of indices $A\subset \set{0,1,\ldots, n - 1}$, we define the \emph{Bowen distance between $x,y\in X$ along $A$} by
$d_A(x, y) = \max\set{d(f^j(x), f^j(y)) : j \in A}$ and the \emph{Bowen ball (of radius $\eps$ centered at $x\in X$) along $A$}
by $B_A(x, \eps) = \set{y \in X : d_A(x, y) < \eps}$.
When $g$ is a mistake function and $(n,\eps)$ is such that $g(n,\eps) < n$, we define for $x\in X$ a \emph{$(g;n,\eps)$-Bowen ball of radius $\eps$, center $x$, and length $n$} 
by
$$
B_n(g; x, \eps) := \bigg\{y \in X : y \in B_A(x, \eps) \text{ for some }A\in I(g;n,\eps)\bigg\}
=
\bigcup_{A\in I(g;n,\eps)}
B_A(x, \eps).
$$

Using the above notation, we are able to present the definition of the almost specification property.

\begin{defn}
A dynamical system $(X,f)$ has the \emph{almost specification
property} if there exists a mistake function $g$ and a function $k_g\colon (0,\infty)\longrightarrow \N$ such that
for any $m\geq 1$, any $\eps_1,\ldots,\eps_m > 0$, any points $x_1, \ldots, x_m \in X$, and any integers
$n_1 \geq k_{g}(\eps_1),\ldots,n_m \geq k_{g}(\eps_m)$ setting $n_0=0$ and
$$
l_j=\sum_{s=0}^{j-1}n_s,\,\text{for }j=1, \ldots, m,
$$
one can find a point $z\in X$ such that for every $j=1, \ldots, m$,
$$
f^{l_j}(z)\in B_{n_j}(g;x_j,\eps_j).
$$
In other words,
the appropriate part of the orbit of $z$, $\eps_j$-traces with at most $g(\eps_j,n_j)$, mistakes the orbit of $x_j$,
 $j=1, \ldots, m$.
\end{defn}

\section{$\M^{\alpha}$ and $\M_{\alpha}$-shadowing properties}\label{sec:Ma}
In this section, we prove that both $\mathscr{M}^{\alpha}$ and
$\mathscr{M}_{\alpha}$ are preserved under iterations. As a corollary,
we show that $\overline{d}$-shadowing implies chain mixing under
the assumption of surjection.

The following proposition has a simple proof, which we leave to the reader.

\begin{prop}\label{Proposition 4.1}
If $(X, f)$ is topologically conjugate to $(Y, g)$ then $f$ has $\mathscr{M}^{\alpha}$-shadowing
property or $\mathscr{M}_{\alpha}$-shadowing property for some $\alpha\in \left[0, 1\right)$,
if and only if $g$ does so.
\end{prop}
\begin{lem}\label{Lemma 2.1}
Let $(X, f)$ be a dynamical system and $\alpha\in \left[0, 1\right)$.
If $f$ has the $\mathscr{M}^{\alpha}$-shadowing property, then $f^{k}$
has the $\mathscr{M}^{\alpha}$-shadowing property for any $k\in \mathbb{N}$.
\end{lem}
\begin{proof}
Given any fixed $\varepsilon>0$, the uniform continuity of $f$
implies that there exists $\gamma\in \left(0,
\varepsilon/4\right)$ such that, for all $x, y\in X$,
\begin{equation}\label{4.1}
d(x, y)<\gamma \quad \Longrightarrow \quad
d(f^{i}(x),
f^{i}(y))<\frac{\varepsilon}{4} \text{ for }i=0,1,\ldots, k.
\end{equation}

There exists $\delta\in (0, \gamma)$ such that every $\delta$-ergodic
pseudo-orbit of $f$ is $\mathscr{M}^{\alpha}$
-$\gamma$-shadowed by some point in $X$.

Fix any $\delta$-ergodic pseudo-orbit $\left\{e_{i}\right\}_{i=0}^{\infty}$ of $f^{k}$,
and let
\[
x_{ik+j}=f^{j}(e_{i}), \ \forall \ i\in \Zp,
0\leq j<k.
\]
Clearly,
$\set{x_{i}}_{i=0}^{\infty}$ is a
$\delta$-ergodic pseudo-orbit of $f$.
Then, there exists $z\in X$ such that
\begin{equation}\label{4.2}
\limsup_{n\rightarrow\infty}\frac{1}{n}\left|\Lambda_{n}(z, \set{x_{i}}_{i=0}^{\infty},
f, \gamma)\right|
>\alpha.
\end{equation}

It suffices to show that
\begin{equation}
\limsup_{n\rightarrow\infty}\frac{1}{n}\left|\Lambda_{n}
(z, \set{e_{i}}_{i=0}^{\infty}, f^{k}, \varepsilon)\right|
>\alpha.\label{calim:4.3}
\end{equation}
Suppose on the contrary that \eqref{calim:4.3} does not hold.
Then
\begin{eqnarray*}
\xi&:=&\liminf_{n\rightarrow\infty}\frac{1}{n}\left|\Lambda_{n}^{c}(z,
\set{e_{i}}_{i=0}^{\infty}, f^{k}, \varepsilon)\right|\\
&=&1-\limsup_{n\rightarrow\infty}\frac{1}{n}\left|\Lambda_{n}(z,
\set{e_{i}}_{i=0}^{\infty}, f^{k}, \varepsilon)\right|
\geq 1-\alpha.
\end{eqnarray*}
For any fixed $Q\in \N$, we can find $N_{Q}\in \N$
such that for any $n\geq N_{Q}$,
\begin{equation}\label{4.4}
\frac{1}{n}\left|\Lambda_{n}^{c}(z, \set{e_i}_{i=0}^\infty, f^{k},
\varepsilon)\right|\geq\xi-\frac{1}{2Q}.
\end{equation}
By the definition of $\delta$-ergodic pseudo-orbit, we have
$$
\lim_{n\rightarrow\infty}\frac{1}{n}\left|\Lambda_{n}(\set{e_i}_{i=0}^\infty,
f^{k}, \delta)+1 \right|=1.
$$
Combining this with \eqref{**} and \eqref{4.4}, it follows that
$$
\liminf_{n\rightarrow\infty}\frac{1}{n}
\left|\left(\Lambda_{n}(\set{e_i}_{i=0}^\infty, f^{k}, \delta)+1\right) \cap
\Lambda_{n}^{c}(z, \set{e_i}_{i=0}^\infty, f^{k}, \varepsilon)\right|
\geq\xi-\frac{1}{2Q},
$$
which implies that there exists $M_{Q}>N_Q$
such that for any $n\geq M_{Q}$,
$$
\frac{1}{n}\left|\left(\Lambda_{n}(\set{e_i}_{i=0}^\infty, f^{k}, \delta)+1\right)
\cap \Lambda_{n}^{c}(z, \set{e_i}_{i=0}^\infty, f^{k},
\varepsilon)\right|\geq\xi-\frac{1}{Q}.
$$
For each $n>M_Q$, denote
$$
\Omega_{n}:=\left(\Lambda_{n}(\set{e_i}_{i=0}^\infty, f^{k},
\delta)+1\right)\cap \Lambda_{n}^{c}(z, \set{e_i}_{i=0}^\infty, f^{k}, \eps).
$$
Observe that if $i\in \Omega_n$ then for each $1\leq j<k$ we have
$f^{j}(x_{ik-j})=f^{k}(x_{(i-1)k})=f^{k}(e_{(i-1)})$ which, combined with the fact that $i-1\in
\Lambda_{n}(\set{e_i}_{i=0}^\infty, f^{k}, \delta)$, gives $d(f^{j}(x_{ik-j}), x_{ik})
<\delta$.
This together with $i\in \Lambda_{n}^{c}(z, \{e_{i}\}_{i=0}^{\infty},
f^{k}, \eps)$ implies that
\begin{eqnarray*}
\varepsilon &\leq& d(f^{ik}(z), x_{ik})\leq
d(f^{j}(f^{ik-j}(z)), f^{j}(x_{ik-j}))+d(f^{j}(x_{ik-j}), x_{ik})\\
&<&d(f^{j}(f^{ik-j}(z)), f^{j}(x_{ik-j}))+\delta,
\end{eqnarray*}
which gives
$$
d(f^{j}(f^{ik-j}(z)),
f^{j}(x_{ik-j}))>\varepsilon-\delta>\frac{3\varepsilon}{4}.
$$
This together with \eqref{4.1} implies that if we fix any $n>M_Q$, any $i\in \Omega_{n}$ and any $1\leq j<k$, then
\begin{equation}\label{4.5}
d(f^{ik-j}(z), x_{ik-j})\geq\gamma.
\end{equation}
Hence, for any $n\geq M_{Q}$,
$$
\Lambda_{nk}^{c}(z, \set{x_i}_{i=0}^\infty, f, \gamma)\supset
\bigcup_{j=0}^{k-1}\left(k\cdot\Omega_{n}-j\right).
$$
Clearly, $(k\cdot\Omega_{n}+s)\cap (k\cdot\Omega_{n}+t)=\emptyset$ for all $0\leq s < t <k$, so
$
\left|\cup_{j=0}^{k-1}\left(k\cdot\Omega_{n}-j\right)\right|=
k\left|\Omega_{n}\right|,
$
and consequently for every $n\geq M_{Q}$ we obtain that
$$
\frac{\left|\Lambda_{nk}^{c}(z, \set{x_i}_{i=0}^\infty, f,
\gamma)\right|}{nk}\geq \frac{\left|\left(\Lambda_{n}(\set{e_i}_{i=0}^\infty,
f^{k}, \delta)+1\right)\cap \Lambda_{n}^{c}(z, \set{e_i}_{i=0}^\infty, f^{k},
\varepsilon)\right|}{n} \geq\xi-\frac{1}{Q}.
$$
This implies that, for any $m\geq k M_{Q}$ and $s\geq M_Q$ such that $sk\leq m < (s+1)k$, the following condition holds:
\begin{eqnarray*}
\frac{\left|\Lambda_{m}^{c}(z, \set{x_i}_{i=0}^\infty, f, \gamma)\right|}{m}
&\geq& \frac{\left|\Lambda_{sk}^{c} (z, \set{x_i}_{i=0}^\infty, f,
\gamma)\right|}{(s+1)k}
\geq\frac{\left|\Lambda_{(s+1)k}^{c}
(z, \set{x_i}_{i=0}^\infty, f, \gamma)\right|-k}{(s+1)k}\\
&\geq&\frac{\left|\Lambda_{(s+1)k}^{c} (z, \set{x_i}_{i=0}^\infty, f,
\gamma)\right|}{(s+1)k}
-\frac{1}{s+1}\\
&\geq& \xi-\frac{1}{Q}-\frac{k}{m}.
\end{eqnarray*}
This immediately implies that
$$
\liminf_{n\rightarrow\infty}\frac{1}{n} \left|\Lambda_{n}^{c}(z,
\set{x_i}_{i=1}^\infty, f, \gamma)\right|\geq\xi-\frac{1}{Q}.
$$
which, since $Q$ can be arbitrarily large, implies that
$$
\liminf_{n\rightarrow\infty}\frac{1}{n} \left|\Lambda_{n}^{c}(z,
\set{x_i}_{i=1}^\infty, f, \gamma)\right|\geq\xi\geq 1-\alpha.
$$
As a consequence of the above observations, we obtain that
\begin{eqnarray*}
\limsup_{n\rightarrow\infty}\frac{1}{n} \left|\Lambda_{n}(z, \set{x_i}_{i=1}^\infty,
f, \gamma)\right| &=&1-\liminf_{n\rightarrow\infty}
\frac{1}{n}\left|\Lambda_{n}^{c}(z, \set{x_i}_{i=1}^\infty, f,
\gamma)\right|\\
&\leq& 1-\xi\leq \alpha,
\end{eqnarray*}
which contradicts \eqref{4.2}.
Thus, \eqref{calim:4.3} holds, which completes the proof.
\end{proof}
\begin{lem}\label{Lemma 2.2}
Let $(X, f)$ be a dynamical system and $\alpha\in [0, 1)$.
If $f^{k}$ has the $\mathscr{M}^{\alpha}$-shadowing property for
some $k\in \mathbb{N}$, then $f$ has the $\mathscr{M}^{\alpha}$-shadowing
property.
\end{lem}

\begin{proof}
Since $f$ is uniformly continuous,
for any $\eps>0$ we can find
$\gamma\in (0, \eps/2)$ such that if $d(f(x_i),x_{i+1})<\gamma$
for all $i=0, \ldots, k$ and $d(z, x_0)<\gamma$
then $d(f^i(z),x_i)<\eps$ for all $i=0, \ldots, k$.
%
%
%

Map $f^{k}$ has the $\mathscr{M}^{\alpha}$-shadowing property, so there
exists $\delta\in (0, \gamma/4)$ such that every
$\delta$-ergodic pseudo-orbit of $f^{k}$ is
$\mathscr{M}^{\alpha}$-$\gamma/ 4$-shadowed by a point
in $X$. Similarly, we can find
$\delta'\in (0, \delta/2)$ such that if $d(f(x_i),x_{i+1})<\delta'$
for all $i=0, \ldots, k$ and $d(z, x_0)<\delta'$
then $d(f^i(z),x_i)<\delta$ for all $i=0, \ldots, k$.

Given any $\delta'$-ergodic pseudo-orbit
$\set{x_{i}}_{i=0}^{\infty}$ of $f$, denote $z_i=x_{ik}$ for all $i\geq 0$.
Denote $\mathscr{C}_{n}:= \bigcap_{j=0}^{k}\left\{0\leq i<n: d(f(x_{ik+j}), x_{ik+j+1})<\delta'
\right\}$ and let  $\mathscr{C}=\cup_{n\in \N}\mathscr{C}_{n}$.
By definition, $\lim_{n\rightarrow\infty}\frac{1}{n}|\Lambda_{n}(\{x_{i}\}_{i=0}
^{\infty}, f, \delta')|=1$ and it is also easy to see that for any $j\in \{0, \ldots,
k\}$,
$$
\lim_{n\rightarrow\infty}
\frac{1}{n}|\left\{0\leq i<n: d(f(x_{ik+j}), x_{ik+j+1})<\delta'\right\}|=1.
$$
By \eqref{**}, we obtain that $\lim_{n\rightarrow\infty}\frac{1}{n}|\mathscr{C}_{n}|=1$.
If $i\in \mathscr{C}$ then $d(f(x_{ik+j}), x_{ik+j+1})<\delta'$ for all $j=0, \ldots, k$, and so,
by the choice of $\delta'$, we immediately obtain that $d(f^{k}(z_{i}), z_{i+1})=
d(f^{k}(x_{ik}), x_{ik+k})<\delta$.
Hence, $\set{z_{i}}_{i=0}^{\infty}$ is a $\delta$-ergodic pseudo-orbit
of $f^{k}$, and consequently, there exists $z\in X$ such that
$$
\limsup_{n\rightarrow\infty}\frac{1}{n} \left|\Lambda_{n}(z,
\set{z_i}_{i=0}^\infty, f^{k}, \gamma/ 4)\right|>\alpha.
$$
In particular, there exist $\xi>\alpha$ and a
strictly increasing sequence $\left\{n_{l}\right\}_{l=1}^{\infty}
\subset\mathbb{N}$ such that for any $l\in \mathbb{N}$ the following condition holds:
\begin{equation}\label{4.6}
\frac{1}{n_{l}} \left|\Lambda_{n_{l}}(z, \set{z_i}, f^{k},
\gamma/4)\right|\geq \xi.
\end{equation}
For $l\in \N$, define
\begin{eqnarray*}
\mathscr{Q}_{n_{l}}&:=& \bigcap_{i=0}^{k-1}\left(\Lambda(\set{x_i}_{i=0}^\infty, f,
\delta')-i \right) \cap k\cdot \Lambda_{n_{l}}(z,
\set{z_i}_{i=1}^\infty, f^{k}, \gamma/4),\\
\mathscr{H}_{l}&:=&[0, l)\cap \bigcap_{i=0}^{k-1} \left(\Lambda(\{x_{i}\}_{i=0}
^{\infty}, f, \delta')-i\right).
\end{eqnarray*}
The above definition can be written as
\begin{equation}
\mathscr{Q}_{n_{l}}=\mathscr{H}_{n_{l}}\cap k\cdot \Lambda_{n_{l}}(z,
\set{z_i}_{i=1}^\infty, f^{k}, \gamma/4).\label{eq:QiH}
\end{equation}

Recall that $\lim_{n\rightarrow\infty}\frac{1}{n}\left|\Lambda_{n}(\{x_{i}\}_{i=0}^{\infty},
f, \delta')\right|=1$ and, therefore,
\begin{equation}\label{4.8}
\lim_{n\rightarrow\infty}\frac{1}{n}\left|\mathscr{H}_{n}\right|=1.
\end{equation}

For $j\in \mathscr{Q}_{n_{l}}$, $d(f^{j}(z),x_{j})
=d((f^{k})^{j/k}(z), z_{j/k})<\gamma/4$ and $d(f(x_{j+i}),x_{j+i+1})<\delta'
<\gamma$ for all $i=0, \ldots, k-1$. Then, by the choice of $\gamma$, we obtain
that $d(f^{j+i}(z),x_{j+i})<\eps$ for all $i=0, \ldots, k-1$
and, as a direct consequence, we obtain that
\begin{equation}\label{4.7}
\Lambda_{(n_{l}+1)k}(z, \{x_{i}\}_{i=0}^{\infty}, f,
\eps)\supset\bigcup_{i=0}^{k-1}
\left(\mathscr{Q}_{n_{l}}+i\right).
\end{equation}
Combining \eqref{eq:QiH},\eqref{4.8} and \eqref{**}, we get that for any $0\leq i<k$,
\[
\begin{split}
\liminf_{l\rightarrow\infty}&\frac{1}{(n_{l}+1)k}
\left|\left(\mathscr{Q}_{n_{l}}+i\right)\right|\\
&=\liminf_{l\rightarrow\infty}\frac{1}{(n_{l}+1)k}\left|\left(k\cdot
\Lambda_{n_{l}}(z, \{z_{i}\}_{i=0}^{\infty}, f^{k}, \gamma/ 4)+i\right)\right|\\
&\geq
\liminf_{l\rightarrow\infty}\frac{n_{l}\xi}{(n_{l}+1)k}=\frac{\xi}{k}.
\end{split}
\]
Clearly, $(\mathscr{Q}_{n_{l}}+s)\cap(\mathscr{Q}_{n_{l}}+t)=\emptyset$
for all $0\leq s<t<k$; hence, applying \eqref{4.7} yields
\[
\begin{split}
\limsup_{n\rightarrow\infty}&\frac{1}{n}\left|\Lambda_{n}(z, \{x_{i}\}_{i=0}^{\infty},
f, \eps)\right|\\
&\geq\liminf_{l\rightarrow\infty}
\frac{1}{(n_{l}+1)k}\left|\Lambda_{(n_{l}+1)k}(z, \{x_{i}\}_{i=0}^{\infty}, f, \eps)\right|\\
&\geq\liminf_{l\rightarrow\infty}\frac{1}{(n_{l}+1)k}
\sum_{i=0}^{k-1}\left|\left(\mathscr{Q}_{n_{l}}+i\right)\right|\\
&\geq\sum_{i=0}^{k-1}\liminf_{l\rightarrow\infty}\frac{1}{(n_{l}+1)k}
\left|\left(\mathscr{Q}_{n_{l}}+i\right)\right|\geq\xi >\alpha.
\end{split}
\]
Indeed, $(X,f)$ has the $\M^\alpha$-shadowing property. The proof is finished.
\end{proof}

Slightly modifying the proofs of Lemma~\ref{Lemma 2.1} and Lemma~\ref{Lemma 2.2},
we can prove the following two Lemmas.
\begin{lem}\label{Lemma 2.3}
Let $(X, f)$ be a dynamical system and $\alpha\in \left[0, 1\right)$.
If $f$ has the $\mathscr{M}_{\alpha}$-shadowing property, then $f^{k}$
has the $\mathscr{M}_{\alpha}$-shadowing property for any $k\in \mathbb{N}$.
\end{lem}
\begin{lem}\label{Lemma 2.4}
Let $(X, f)$ be a dynamical system and $\alpha\in [0, 1)$.
If $f^{k}$ has the $\mathscr{M}_{\alpha}$-shadowing property for
some $k\in \mathbb{N}$, then $f$ has the $\mathscr{M}_{\alpha}$-shadowing
property.
\end{lem}
Combining together Lemmas \ref{Lemma 2.1}--\ref{Lemma 2.4}, we obtain the following result.
\begin{thm}\label{M_aiterations}
Let $(X, f)$ be a dynamical system and $\alpha\in [0, 1)$.
Then the following statements are equivalent:
\begin{enumerate}
\item\label{3.6.1}
$f$ has the $\mathscr{M}^{\alpha}$-shadowing property (resp.,
$\mathscr{M}_{\alpha}$-shadowing property);

\item\label{3.6.2} $f^{k}$ has the $\mathscr{M}^{\alpha}$-shadowing property (resp.,
$\mathscr{M}_{\alpha}$-shadowing property) for
any $k\in \mathbb{N}$;

\item\label{3.6.3} $f^{k}$ has the $\mathscr{M}^{\alpha}$-shadowing property (resp.,
$\mathscr{M}_{\alpha}$-shadowing property) for
some $k\in \mathbb{N}$.
\end{enumerate}
\end{thm}

\begin{cor}\label{dshad_higher}
For a dynamical system $(X, f)$, the following statements are equivalent:
\begin{enumerate}
\item\label{4.7.1} $f$ has the $\overline{d}$-shadowing property (resp.,
$\underline{d}$-shadowing property);

\item\label{4.7.2} $f^{k}$ has the $\overline{d}$-shadowing property (resp.,
$\underline{d}$-shadowing property) for any $k\in \N$;

\item\label{4.7.3} $f^{k}$ has the $\overline{d}$-shadowing property (resp.,
$\underline{d}$-shadowing property) for some $k\in \N$.
\end{enumerate}
\end{cor}

\begin{cor}\label{Corollary XX}
Let $f: X\longrightarrow X$ be a surjection. If $(X, f)$ has
the $\overline{d}$-shadowing property, then it is chain mixing.
\end{cor}
\begin{proof}
Dynamical system $(X,f^n)$ is chain transitive for every $n\geq 1$, by \cite[Theorem 2.2]{DH2010} and Corollary \ref{dshad_higher}.
So it is chain mixing by \cite[Corollary 12]{RW2008}.
\end{proof}
\begin{exa}\label{Example 4.4}
Let $X=\{a_{1}, a_{2}\}$ be any two distinct points with the discrete metric $d$ and let
$f: X\longrightarrow X$ be the identity map. It is easy to see that for every
sequence $\{x_{i}\}_{i=0}^{\infty}\subset X$, there exists $j\in \{1, 2\}$ such that $\overline{d}\left(\left\{i\in \Zp:
x_{i}=a_{j}\right\}\right)\geq 1/2$. Hence, $\overline{d}(\Lambda(a_{j}, \{x_{i}\}_{i=0}^{\infty},
f, \eps))\geq 1/2$ holds for any $\eps >0$. This means that $f$ has the $\mathscr{M}^{\alpha}$-shadowing
property for any $\alpha\in\left[0, 1/2\right)$. However, it is clear that $f$
is not chain mixing.
\end{exa}

\begin{rem}Example~\ref{Example 4.4}
shows that, in general, the $\mathscr{M}^{\alpha}$-shadowing property does not
imply chain mxing when $\alpha\in \left[0, 1/2\right)$. By Corollary~\ref{Corollary XX}, however,
the $\mathscr{M}_{\alpha}$-shadowing property implies chain mixing, provided that $f$ is surjective and
$\alpha\in\left[1/2, 1\right).$
\end{rem}

\section{AASP implies ASP}
The aim of this section is to prove that the average shadowing
property is a consequence of the asymptotic average shadowing
property. The result works for general cases (no assumption that the
map is onto, see \cite[Theorem 3.7]{KKO14}).

\begin{thm}\label{Theorem 5.1}
If a dynamical system $(X, f)$ has the weak asymptotic average
shadowing property, then it also has the average shadowing property.
\end{thm}
\begin{proof}
Suppose that $f$ does not have the average shadowing property. Then,
there exists $\eps>0$ such that for any $k\in \N$,
there exists a $1/k$-average-pseudo-orbit
$\beta^{(k)}:=\set{\beta^{(k)}_{i}}_{i=0}^{\infty}$ which
is not $\eps$-shadowed on average by any point in $X$, i.e.,
\begin{equation}\label{5.1}
\forall z\in X,\ \ \limsup_{k\rightarrow\infty}
\frac{1}{n}\sum_{i=0}^{n-1}d(f^{i}(z), \beta^{(k)}_{i})\geq \eps.
\end{equation}

Now, we construct an asymptotic average pseudo-orbit $\xi$
which is not $\eps/2$-shadowed on average by any
point in $X$. For every $k\in \N$, the sequence $\beta^{(k)}$ is a
 $1/k$-average-pseudo-orbit; hence, there exists $N_{k}\in \N$
 such that for all integers $n\geq N_{k}$ and $i\geq 0$,
\begin{equation}\label{5.2}
\frac{1}{n}\sum_{j=0}^{n-1}d(f(\beta_{j+i}^{(k)}), \beta_{j+i+1}^{(k)})<\frac{1}{k}.
\end{equation}
Clearly, we may assume that $\set{N_{k}}_{k=1}^{\infty}$
is a strictly increasing sequence.

Put $m_{1}=2^{N_{2}}$ and define inductively a sequence $m_{2},
m_{3}, \ldots$ in the following way. Suppose that we have already
defined $m_n$ for some $n\geq 1$. Observe that by \eqref{5.1}, for
any $z\in X$ and any $k,N>0$ there exist $l>N$
and $\eta>0$ such that if $d(z,y)<\eta$ then
$$\frac{1}{l+1}\sum_{i=0}^{l}d(f^{i}(y), \beta^{(k)}_{i})\geq \frac{\eps}{2}.$$
In particular, by compactness of $X$ there exist
$k_{n+1}\in \N$ and positive integers
$$
L^{(n+1)}_{1}, L^{(n+1)}_{2}, \ldots, L^{(n+1)}_{k_{n+1}}\geq 2^{(n+1)m_n},
$$
such that for any $z\in X$ there exists
$1\leq i \leq k_{n+1}$ such that
\begin{equation}\label{5.3}
\frac{1}{L^{(n+1)}_{i}+1}\sum_{j=0}^{L^{(n+1)}_{i}}d(f^{j}(z),
\beta_{j}^{(n+1)})\geq \frac{\eps}{2}.
\end{equation}
Denote
$$
m_{n+1}=\max\left\{2^{N_{n+2}}, L^{(n+1)}_{1},
L^{(n+1)}_{2}, \ldots, L^{(n+1)}_{k_{n+1}}\right\}.
$$
And consider the sequence
$$
\xi=\left\{\xi_{i}\right\}_{i=0}^{\infty}=\beta_{0}^{(1)}
\beta_{1}^{(1)}\cdots\beta_{m_{1}}^{(1)}
\beta_{0}^{(2)}\beta_{1}^{(2)}\cdots
\beta_{m_{2}}^{(2)}\cdots\beta_{0}^{(n)}
\beta_{1}^{(n)}\cdots\beta_{m_{n}}^{(n)} \cdots.
$$
We claim that $\xi$ is an asymptotic average pseudo-orbit of $f$.
Denote $M_{0}=0$ and for $n>0$ put $M_{n}=\sum_{i=1}^n (m_i+1)$.
Observe that $M_n\leq n(m_n+1)\leq (n+1)m_n$ and hence $M_{n+1}\geq
m_{n+1} \geq 2^{M_n}$. Similarly $M_n\geq 2^{N_{n+1}}> N_{n+1}$.

Fix an $n>0$ and an integer $j\in [M_n,M_{n+1})$.
By the definition of $\xi$ we obtain
\begin{eqnarray*}
&&\frac{1}{j}\sum_{i=0}^{j-1} d(f(\xi_{i}), \xi_{i+1})\\
&&\quad = \frac{1}{j}\left[
\sum_{k=1}^{n}\sum_{i=M_{k-1}}^{M_{k}-2}d(f(\xi_{i}), \xi_{i+1})+\sum_{i=1}
^{n}d(f(\xi_{M_{i}-1}), \xi_{M_{i}})+
\sum_{i=M_{n}}^{j-1}d(f(\xi_{i}), \xi_{i+1})\right].
\end{eqnarray*}
Note that
\begin{eqnarray*}
&&\frac{1}{j}
\sum_{k=1}^{n}\sum_{i=M_{k-1}}^{M_{k}-2}d(f(\xi_{i}), \xi_{i+1})
+\frac{1}{j}\sum_{i=1}
^{n}d(f(\xi_{M_{i}-1}), \xi_{M_{i}})\\
&&\quad\quad  \leq
\sum_{k=1}^{n}\frac{m_{k}}{j}
\frac{1}{m_{k}}\sum_{i=M_{k-1}}^{M_{k}-2}d(f(\xi_{i}), \xi_{i+1})
+\frac{n \diam X}{j}\\
&&\quad\quad  \leq \sum_{k=1}^{n}\frac{m_{k}}{j k}+\frac{n \diam X}{2^n}
\leq \frac{1}{n}+\sum_{k=1}^{n-1}\frac{m_{k}}{j k}+\frac{n \diam X}{2^n}\\
&&\quad\quad  \leq \frac{1}{n}+\frac{M_{n-1}}{M_n}+\frac{n \diam X}{2^n}
\leq \frac{1}{n}+\frac{M_{n-1}}{2^{M_{n-1}}}+\frac{n \diam X}{2^n}.
\end{eqnarray*}
Additionally, observe that if $j\leq M_n + N_{n+1}$ then
$$
\frac{1}{j}\sum_{i=M_{n}}^{j-1}d(f(\xi_{i}), \xi_{i+1})\leq
\frac{N_{n+1}}{j}\diam X\leq \frac{N_{n+1}}{2^{N_{n+1}}}\diam X,
$$
and in the second case of $j> M_n + N_{n+1}$, by the choice of
$N_{n+1}$ we immediately obtain that
$$
\frac{1}{j}\sum_{i=M_{n}}^{j-1}d(f(\xi_{i}), \xi_{i+1})\leq
\frac{1}{j}\sum_{i=M_{n}}^{j-1}d(f(\beta^{(n+1)}_{i-M_n}), \beta^{(n+1)}_{i-M_n+1})
\leq \frac{1}{n+1}.
$$
We have just proved that $\lim_{j\to \infty} \frac{1}{j}
\sum_{i=0}^{j-1} d(f(\xi_{i}), \xi_{i+1})=0$, so indeed $\xi$ is an
asymptotic average pseudo-orbit and the claim holds.

Fix $z\in X$. By \eqref{5.3}, for any $n\in \mathbb{N}$ and
point $f^{M_n}(z)$ we can select $1\leq i_{n} \leq k_{n+1}$ such
that
\[
\frac{1}{L^{(n+1)}_{i_{n}}+1}
\sum_{j=M_{n}}^{M_{n}+L_{i_n}^{(n+1)}} d(f^{j}(z), \xi_{j})=\frac{1}{L^{(n+1)+1}_{i_n}}
\sum_{j=0}^{L^{(n+1)}_{i_n}}d(f^{j}(f^{M_{n}}(z)),
\beta_{j}^{(n+1)})\geq \frac{\eps}{2}.
\]
Therefore,
\begin{eqnarray*}
&&\limsup_{n\rightarrow\infty}\frac{1}{n}
\sum_{j=0}^{n-1}d(f^{j}(z), \xi_{j})\\
&&\quad\quad\ \geq\ \limsup_{n\rightarrow\infty}
\frac{1}{M_{n}+L_{i_n}^{(n+1)}+1}
\sum_{j=0}^{M_{n}+L_{i_n}^{(n+1)}}d(f^{j}(z), \xi_{j})\\
&&\quad\quad\ \geq\ \limsup_{n\rightarrow\infty}
\frac{L^{(n+1)}_{i_n}+1}{M_{n}+L_{i_n}^{(n+1)}+1}\frac{1}{L^{(n+1)}_{i_n}+1}
\sum_{j=M_{n}}^{M_{n}+L_{i_n}^{(n+1)}}d(f^{j}(z), \xi_{j})\\
&&\quad\quad\ \geq\ \frac{\eps}{2}\limsup_{n\rightarrow\infty}
\frac{L^{(n+1)}_{i_n}+1}{M_{n}+L_{i_n}^{(n+1)}+1}
\\
&&\quad\quad\ \geq\ \frac{\eps}{2}\limsup_{n\rightarrow\infty}
\frac{2^{M_{n}}}{M_{n}+2^{M_{n}}+1}= \frac{\eps}{2}.
\end{eqnarray*}
This means that $\xi$ is not $\eps/2$-shadowed on
average by any point in $X$. Hence, $f$ does not have
the weak average shadowing property.
\end{proof}
\begin{thm}\label{Theorem 5.2}
If a dynamical system $(X, f)$ has the asymptotic average
shadowing property, then it also has the average shadowing property.
\end{thm}\label{Theorem 5.3}
\begin{proof}
It directly follows from the definition that the asymptotic average
shadowing property implies the weak asymptotic average shadowing
property, hence the result immediately follows by
Theorem~\ref{Theorem 5.1}.
\end{proof}

Careful readers can check that a slight change in the proof of Theorem~\ref{Theorem 5.1} leads to the following
theorem.
\begin{thm}
If a dynamical system $(X, f)$ has the weak asymptotic average
shadowing property, then for any $\eps>0$, there exists $\delta>0$
such that every $\delta$-asymptotic-average-pseudo-orbit is $\eps$-shadowed
on average by some point in $X$.
\end{thm}

\section{Further studies on the $\mathscr{M}_{\alpha}$-shadowing property and ASP}\label{sec:equivASP}

In this section, we continue our investigation of the $\mathscr{M}_{\alpha}$-shadowing property, and as a byproduct also
the average shadowing property. Recently, it was proved in \cite[Theorem 5]{ODH2014} that for a surjection, the
average shadowing property implies the $\underline{d}$-shadowing
property and the authors also provided an example \cite[Example 20]{ODH2014} showing that
this implication can not be reversed.

In what follows, we provide a few equivalent conditions to the $\mathscr{M}_{\alpha}$-shadowing property (see Theorem \ref{Theorem 6.5}),
which will reveal the reasons why \cite[Theorem 5]{ODH2014} actually holds and why there is no chance for reverse implication.
It was proved in \cite[Theorem 3.6]{KKO14} that, under the
assumption of chain mixing, the average shadowing property is all
about average shadowing of pseudo-orbits.

Wherever we cannot guarantee chain mixing, the situation is more complex.
Still, Theorem \ref{Theorem 6.5} allows us to pass from average pseudo-orbits to asymptotic average pseudo-orbits
shadowed with the spatial scale (Theorem \ref{Theorem 6.5}~\eqref{6.5.3}) or
ergodic pseudo-orbits shadowed with the time scale (Theorem
\ref{Theorem 6.5}~\eqref{6.5.2}). With this tool at hand, we are able to characterize the average shadowing property
equivalently from both space (weak asymptotic average shadowing
property) and time ($\M_{\alpha}$-shadowing property) perspectives.

Before we prove the main results of this section, we need the following.

\begin{lem}\label{Lemma 6.1}
Let $(X, f)$ be a dynamical system. Then, for any $\delta>0$ and any
$\delta/2$-ergodic pseudo-orbit $\set{x_{i}}_{i=0}^{\infty}$ of $f$,
there exists a $\delta$-average-pseudo-orbit
$\left\{y_{i}\right\}_{i=0}^{\infty}$ of $f$ such that
$d\left(\left\{i\in \mathbb{N}_{0}: x_{i}\neq
y_{i}\right\}\right)=0$.
\end{lem}
\begin{proof}
Fix a $\delta>0$ and a $\delta/2$-ergodic pseudo-orbit
$\set{x_{i}}_{i=0}^{\infty}$. Take any sufficiently large positive
integer $N$ to ensure that $4\diam X/N <\delta/2$. If there is $n>0$
such that $\set{x_{i}}_{i=n}^{\infty}$ is a $\delta/2$-pseudo orbit
then clearly $\set{x_{i}}_{i=0}^{\infty}$ is a
$\delta$-average-pseudo-orbit and there is nothing to prove. Hence,
assume that $\Lambda^c(\set{x_i}_{i=0}^\infty,\delta/2)$ is infinite
and let $\set{n_i}_{i=1}^\infty$ be a strictly increasing sequence
such that $\Lambda^c(\set{x_i}_{i=0}^\infty,\delta/2)=\set{n_i :
i\geq 1}$. We put $k_1=n_1$ and inductively define a sequence
$\set{k_i}_{i=1}^\infty$ by the formula
$$
k_{n+1}=\min \set{j\in \Lambda^{c}(\set{x_i}_{i=0}^\infty, \delta/2): j\geq k_{n}+N},
$$
where $n=1,2,\ldots$. Let $K=\set{k_{n}: n\in \N}$ and $\K=\cup_{i=0}^{N-1}(K+i)$.
It is not hard to verify that the following conditions are satisfied:
\begin{enumerate}[(i)]
\item sets $K$, $K+1$, $\ldots$, $K+(N-1)$ are mutually disjoint;

\item $\Lambda^{c}(\set{x_i}_{i=0}^\infty, \delta/2)=\set{n_{1}, n_{2}, \ldots}\subset\K$;

\item $d(\K)=0$.
\end{enumerate}

Let $\set{y_{i}}_{i=0}^{\infty}$ be a sequence defined by
$$
y_{i}=\begin{cases}
f^{i-k_{n}}(x_{k_{n}}), & \text{ when } i\in \left[k_{n}, k_{n}+N\right)
\text{ for some } n\in \N,\\
x_{i}, & \text{ otherwise},
\end{cases}
$$
and denote (for integers $k\geq 0$, $n>0$)
$$
A^n_k=\left\{i\in \left[k, k+n\right): d(f(y_{i}), y_{i+1})\geq
\frac{\delta}{2}\right\}.
$$
Observe that $|A_k^n |\leq 2\left(\frac{n}{N}+1\right)$.

We claim that $\set{y_i}_{i=0}^\infty$ is a
$\delta$-average-pseudo-orbit. To prove the claim, fix
an $n\geq N$ and a $k\geq 0$. First, observe that if $[k,
k+n)\cap \mathscr{K}=\emptyset$, then
$$
\frac{1}{n}\sum_{i=0}^{n-1} d(f(y_{i+k}), y_{i+k+1})<\frac{\delta}{2}<\delta.
$$
But if $[k, k+n)\cap \mathscr{K}\neq\emptyset$ then
$$
\left|\left\{i\in \left[k, k+n\right): d(f(y_{i}), y_{i+1})\geq
\frac{\delta}{2}\right\}\right|=\left|A_k^n\right|\leq \frac{2(n+N)}{N}\leq \frac{4n}{N},
$$
and hence
\begin{eqnarray*}
\frac{1}{n}\sum_{i=0}^{n-1}
d(f(y_{i+k}), y_{i+k+1})&=&\frac{1}{n}\sum_{i\in A_k^n}d(f(y_{i}), y_{i+1})+
\frac{1}{n}\sum_{i\in \left[k, k+n\right)\setminus A_k^n}d(f(y_{i}), y_{i+1})\\
&\leq&\frac{|A_k^n|}{n}\diam X+\frac{\delta}{2}\leq \frac{4}{N}\diam X+\frac{\delta}{2}\\
&<&\delta.
\end{eqnarray*}
This proves the claim and ends the proof at the same time, since $\set{i\in \N : x_i\neq y_i}\subset \K$.
\end{proof}

\begin{thm}\label{Theorem 6.2}
Let $(X, f)$ be a dynamical system and $\alpha\in \left[0, 1\right)$.
The following statements are equivalent:
\begin{enumerate}
\item\label{6.2.1} $(X, f)$ has the $\mathscr{M}_{\alpha}$-shadowing
property;

\item\label{6.2.2} for every $\eps>0$ and every
asymptotic average pseudo-orbit $\left\{x_{i}\right\}_{i=0}^{\infty}$
of $f$, there exists $z\in X$ such that $\left\{x_{i}\right\}_{i=0}^{\infty}$
is $\mathscr{M}_{\alpha}$-$\eps$-shadowed
by $z$;

\item\label{6.2.3} for every $\eps>0$, there exists $\delta>0$ such that every
$\delta$-average-pseudo-orbit of $f$ is $\mathscr{M}_{\alpha}$-$\eps$-shadowed
by some point in $X$;

\item\label{6.2.4} for every $\eps>0$, there exists $\delta>0$ such that every
$\delta$-asymptotic-average-pseudo-orbit of $f$ is $\mathscr{M}_{\alpha}$-$\eps$-shadowed
by some point in $X$.
\end{enumerate}
\end{thm}

\begin{proof}
Clearly every asymptotic average pseudo-orbit is $\delta$-ergodic pseudo-orbit for any
$\delta>0$ and every $\delta$-average-pseudo-orbit is a $2\delta$-asymptotic-average-pseudo-orbit,
hence $\eqref{6.2.1}\Longrightarrow \eqref{6.2.2}$ and $\eqref{6.2.4}\Longrightarrow
\eqref{6.2.3}$. By Lemma~\ref{Lemma 6.1},
we obtain $\eqref{6.2.3}\Longrightarrow \eqref{6.2.1}$.

First, we show $\eqref{6.2.2}\Longrightarrow \eqref{6.2.1}$. Suppose on the contrary that $f$ does not have the $\mathscr{M}_{\alpha}$-shadowing property.
Then, there exists $\varepsilon>0$ such that for any $k\in \N$, there exists a
$1/2k$-ergodic pseudo-orbit $\{\beta_{i}^{(k)}\}_{i=0}^{\infty}$ which is not
$\mathscr{M}_{\alpha}$-$\varepsilon$-shadowed by any point in $X$. By Lemma~\ref{Lemma 6.1},
for any $k>0$ there exists a $1/k$-average-pseudo-orbit
$\set{\lambda_{i}^{(k)}}_{i=0}^{\infty}$ such that $d(\{i\in \Zp: \beta^{(k)}_{i}\neq \lambda^{(k)}_{i}\})=0$.
Note that for every $z\in X$ we have
\begin{eqnarray*}
\limsup_{n\rightarrow\infty}\frac{1}{n}\left|\Lambda^{c}_{n}(z, \set{\lambda^{(k)}_i}_{i=0}^\infty, f, \varepsilon)\right|
&=&1-\liminf_{n\rightarrow\infty}\frac{1}{n}\left|\Lambda_{n}(z, \set{\lambda^{(k)}_i}_{i=0}^\infty, f, \varepsilon)\right|\\
&\geq& 1-\alpha.
\end{eqnarray*}

Repeating the argument in the proof of Theorem \ref{Theorem 5.1},
we can find an increasing sequence $\left\{l_{n}\right\}_{n=1}^{\infty}$
such that if we denote $\mathscr{L}_{n}=\sum_{j=1}^{n}(l_{j}+1)$ then the following conditions are satisfied:
\begin{enumerate}[(i)]
\item for any $n\geq 2$, there exists $L_{1}^{(n)}, L_{2}^{(n)},
\ldots, L_{k_{n}}^{(n)}\in \left[2^{\mathscr{L}_n},
+\infty\right)$ such that for any $z\in X$, there exists $1\leq j \leq
k_{n}$ satisfying
$$
\frac{1}{L_{j}^{(n)}+1}\left|\Lambda^{c}_{L_{j}^{(n)}+1}(z, \set{\lambda^{(n)}_i}_{i=0}^\infty, f, \varepsilon/2)
\right|\geq 1-\alpha-\frac{1}{2n};
$$

\item $l_{n}\geq \max\left\{L_{1}^{(n)}, L_{2}^{(n)},
\ldots, L_{k_{n}}^{(n)}\right\}$;

\item sequence
$$
\left\{\zeta_{i}\right\}_{i=0}^{\infty}=
\lambda_{0}^{(1)}\lambda_{1}^{(1)}\cdots \lambda_{l_{1}}^{(1)}
\lambda_{0}^{(2)}\lambda_{1}^{(2)}\cdots \lambda_{l_{2}}^{(2)}\cdots
\lambda_{0}^{(n)}\lambda_{1}^{(n)}\cdots \lambda_{l_{n}}^{(n)}\cdots$$
is an asymptotic average pseudo-orbit.
\end{enumerate}

Fix a $z\in X$ and an $n\geq 2$.
There exists $1\leq i_{n}\leq k_{n}$, such that
$$
\frac{1}{L_{i_{n}}^{(n)}+1}\left|\Lambda^{c}_{L_{i_{n}+1}^{(n)}}(f^{\mathscr{L}_{n-1}}(z), \set{\lambda^{(n)}_i}_{i=0}^\infty, f, \varepsilon/2)
\right|\geq 1-\alpha-\frac{1}{2n}.
$$
Using all the above calculations, we obtain that
\begin{eqnarray*}
&&\limsup_{n\rightarrow\infty}\frac{1}{n}\left|\Lambda^{c}_{n}(z, \set{\zeta_i}_{i=0}^\infty, f, \varepsilon/2)\right|\\
&&\quad\quad\quad \geq \limsup_{n\rightarrow\infty}\frac{1}{\mathscr{L}_{n-1}+L_{i_n}^{(n)}+1}
\left|\Lambda^{c}_{\mathscr{L}_{n-1}+L_{i_n}^{(n)}+1}(z, \set{\zeta_i}_{i=0}^\infty, f, \varepsilon/2)\right|\\
&&\quad\quad\quad \geq \limsup_{n\rightarrow\infty}\frac{L^{(n)}_{i_n}+1}{\mathscr{L}_{n-1}+L_{i_n}^{(n)}+1}\frac{1}{L^{(n)}_{i_n}+1}
\left|\Lambda^{c}_{L_{i_{n}}^{(n)}+1}(f^{\mathscr{L}_{n-1}}(z), \set{\lambda^{(n)}_i}_{i=0}^\infty, f, \varepsilon/2)
\right|\\
&&\quad\quad\quad\geq \limsup_{n\rightarrow\infty}\frac{1}{\frac{\mathscr{L}_{n-1}}{L^{(n)}_{i_n}+1}+1}\left(1-\alpha-\frac{1}{2n}\right)
\geq \limsup_{n\rightarrow\infty}\frac{1}{\frac{\mathscr{L}_{n-1}}{2^{\mathscr{L}_{n-1}}}+1}\left(1-\alpha-\frac{1}{2n}\right)\\
&&\quad\quad\quad=
1-\alpha.
\end{eqnarray*}
As a consequence, we obtain that
$$
\liminf_{n\rightarrow\infty}\frac{1}{n}\left|\Lambda_{n}(z, \set{\zeta_i}_{i=0}^\infty, f, \eps/2)\right|
=1-\limsup_{n\rightarrow\infty}\frac{1}{n}\left|\Lambda^{c}_{n}(z, \set{\zeta_i}_{i=0}^\infty, f, \eps/2)\right|
\leq \alpha.
$$
But $z$ was arbitrary, hence there is no point $z\in X$ satisfying \eqref{6.2.2} with respect to the asymptotic average pseudo-orbit $\set{\zeta_i}_{i=1}^\infty$.
This is a contradiction.

Second, to prove $\eqref{6.2.1} \Longrightarrow \eqref{6.2.3}$ suppose that there exists $\varepsilon>0$
such that for any $k\in \mathbb{N}$, there exists a $1/k$-average-pseudo-orbit
of $f$ which is not $\mathscr{M}_{\alpha}$-$\varepsilon$-shadowed by any point in $X$.
 Repeating the argument in the proof of $\eqref{6.2.2}\Longrightarrow \eqref{6.2.1}$, we see that there exists an
asymptotic average pseudo-orbit $\xi=\left\{\xi_{i}\right\}_{i=0}^{\infty}$
of $f$ which is not $\mathscr{M}_{\alpha}$-$\varepsilon/2$-shadowed by any point
in $X$. This contradicts assumption \eqref{6.2.1} as $\xi$ is also a $\delta$-ergodic
pseudo-orbit for any $\delta>0$.

Finally, similarly to the proof of $\eqref{6.2.1} \Longrightarrow \eqref{6.2.3}$, it is not
difficult to prove that $\eqref{6.2.1} \Longrightarrow \eqref{6.2.4}$.
\end{proof}

\begin{lem}\label{Lemma 6.3}
If a dynamical system $(X, f)$ has the average shadowing property,
then it also has the $\M_{\alpha}$-shadowing property for every
$\alpha\in \left[0, 1\right)$.
\end{lem}
\begin{proof}
Fix an $\alpha\in \left[0, 1\right)$, an $\eps>0$ and denote
$\gamma=(1-\alpha)\varepsilon$. There exists $\delta>0$ such that
every $\delta$-average-pseudo-orbit is $\gamma$-shadowed on average
by some point in $X$. Fix a $\delta/2$-ergodic pseudo-orbit
$\set{x_{i}}_{i=0}^{\infty}$ and let $\set{y_{i}}_{i=0}^{\infty}$ be
a $\delta$-average-pseudo-orbit provided by Lemma~\ref{Lemma 6.1}. Then,
by the average shadowing property of $f$, there exists $z\in X$ such
that
\begin{eqnarray*}
\gamma=(1-\alpha)\eps& >&\limsup_{n\rightarrow\infty}\frac{1}{n}\sum
_{i=0}^{n-1}d(f^{i}(z), y_{i})\geq \limsup_{n\rightarrow\infty}
\frac{\eps}{n}\left|\Lambda_{n}^{c}(z, \left\{y_{i}\right\}_{i=0}^{\infty}, \eps)\right|\\
& =&\eps\left(1-\liminf_{n\rightarrow\infty}
\frac{1}{n}\left|\Lambda_{n}(z, \left\{y_{i}\right\}_{i=0}^{\infty},\eps)\right|\right).
\end{eqnarray*}
This immediately implies that
$$
\liminf_{n\rightarrow\infty}
\frac{1}{n}\left|\Lambda_{n}(z, \left\{y_{i}\right\}_{i=0}^{\infty}, \eps)\right|>\alpha.
$$
But, by Lemma~\ref{Lemma 6.1}, if we denote $\K=\set{i : x_i \neq y_i}$ then
 $d(\K)=0$ and, as a consequence,
\begin{eqnarray*}
\liminf_{n\rightarrow\infty}
\frac{1}{n}\left|\Lambda_{n}(z, \set{x_i}_{i=0}^\infty, \eps)\right|&\geq&
\liminf_{n\rightarrow\infty}
\frac{1}{n}\left(\left|\Lambda_{n}(z, \left\{y_{i}\right\}_{i=0}^{\infty}, \eps)\right|-|\K \cap [0,n)|\right)\\
&=&\liminf_{n\rightarrow\infty}
\frac{1}{n}\left|\Lambda_{n}(z, \left\{y_{i}\right\}_{i=0}^{\infty}, \eps)\right|-d(\K)\\
&>&\alpha.
\end{eqnarray*}
Therefore, every $\delta/2$-ergodic pseudo-orbit is
$\M_{\alpha}$-$\varepsilon$-shadowed by a point in $X$. But $\eps$
was arbitrary, so indeed $f$ has the $\mathscr{M}_{\alpha}$-shadowing
property.
\end{proof}

\begin{lem}\label{Lemma 6.4}
If a dynamical system $(X, f)$ has the
$\M_{\alpha}$-shadowing property for every $\alpha\in
\left[0, 1\right)$, then $(X,f)$ has the average shadowing property.
\end{lem}
\begin{proof}
Suppose on the contrary that $(X,f)$ does not have the average shadowing
property. Applying Theorem~\ref{Theorem 5.1} implies that
there exist $\eps>0$ and an asymptotic average
pseudo-orbit $\set{x_{i}}_{i=0}^{\infty}$ such
that
\begin{equation}\label{4.1}
\forall z\in X,\  \ \limsup_{n\rightarrow\infty}\frac{1}{n}\sum_{i=0}^{n-1}
d(f^{i}(z), x_{i})\geq \varepsilon.
\end{equation}
Fix $\alpha\in \left[0, 1\right)$ such that
$(1-\alpha)\diam X<\eps/4$. The
$\M_{\alpha}$-shadowing property of $f$ implies that there
exists $\gamma>0$ such that every
$\gamma$-ergodic pseudo-orbit is
$\M_{\alpha}$-$\eps/4$-shadowed by a point in $X$.
Clearly, $\set{x_i}_{i=0}^\infty$ is a $\delta$-ergodic pseudo-orbit for any
$\delta>0$. By Lemma~\ref{Lemma 6.1}, there
exists a $\gamma$-ergodic pseudo-orbit
$\set{y_{i}}_{i=0}^{\infty}$ such that
$d(\K)=0$, where $\K=\set{i: x_{i}\neq y_{i}}$. Then,
there exists $y \in X$ such that
$\liminf_{n\rightarrow\infty}\frac{1}{n}\left|\Lambda_{n}(y,
\set{y_i}_{i=0}^\infty, \eps/4)\right|>\alpha$. This implies that
there exists $N>0$ such that for any $n\geq N$,
$$
\frac{1}{n}\left|\Lambda_{n}^{c}(y, \set{y_{i}}_{i=0}^{\infty}, \eps/4)\right|=1-
\frac{1}{n}\left|\Lambda_{n}(y, \set{y_{i}}_{i=0}^{\infty}, \eps/4)\right|\leq 1-\alpha,
$$
which leads to the following:
\begin{eqnarray*}
\frac{1}{n}\sum_{i=0}^{n-1}d(f^{i}(y), y_{i})&=&\frac{1}{n}
\sum_{i\in \Lambda_{n}(y, \set{y_{i}}_{i=0}^{\infty}, \eps/4)}d(f^{i}(y), y_{i})+
\frac{1}{n}
\sum_{i\in \Lambda^c_{n}(y, \set{y_{i}}_{i=0}^{\infty}, \eps/4)}d(f^{i}(y), y_{i})\\
&\leq& \frac{\eps}{4}+\frac{\left|\Lambda^c_{n}(y, \set{y_{i}}_{i=0}^{\infty}, \eps/4)\right|}{n}\diam X\\
&\leq& \frac{\eps}{4}+(1-\alpha)\diam X<\frac{\varepsilon}{2}.
\end{eqnarray*}
Now, combining \eqref{4.1} with the fact that $d(\K)=0$ we obtain that
\begin{equation}\label{5.2}
\begin{split}
\eps &\ \leq\ \limsup_{n\rightarrow\infty}
\frac{1}{n}\sum_{i=0}^{n-1}d(f^{i}(y), x_{i})\\
&\ =\ \limsup_{n\rightarrow\infty}\frac{1}{n}\left(\sum_{i\in \left[0, n\right)
\cap \Zp\setminus\mathscr{K}}d(f^{i}(y), y_{i})+\sum_{i\in \left[0, n\right)\cap
\mathscr{K}}d(f^{i}(y), x_{i})\right)\\
&\ \leq\  \limsup_{n\rightarrow\infty}\frac{1}{n}\left(\sum_{i=0}^{n-1}d(f^{i}(y), y_{i})
+\sum_{i\in \left[0, n\right)\cap
\mathscr{K}} \mathrm{diam}X\right)\\
&\ <\  \frac{\eps}{2}+d(\K)\diam X =\frac{\varepsilon}{2},
\end{split}
\end{equation}
which is impossible since $\eps>0$. The proof is completed.
\end{proof}

\begin{thm}\label{Theorem 6.5}
Let $(X, f)$ be a dynamical system. Then the following statements are equivalent:
\begin{enumerate}
\item\label{6.5.1} $f$ has the average shadowing property;

\item\label{6.5.2} $f$ has the $\M_{\alpha}$-shadowing property for every $\alpha\in
\left[0, 1\right)$;

\item\label{6.5.3} $f$ has the weak asymptotic average shadowing property;

\item\label{6.5.4} for any $\alpha\in [0, 1)$, any $\eps>0$ and any
asymptotic average pseudo-orbit $\left\{x_{i}\right\}_{i=0}^{\infty}$
of $f$, there exists $z\in X$ such that $\left\{x_{i}\right\}_{i=0}^{\infty}$
is $\mathscr{M}_{\alpha}$-$\eps$-shadowed by $z$;

\item\label{6.5.5} for any $\alpha\in [0, 1)$ and any $\eps>0$, there exists $\delta>0$ such that every
$\delta$-average-pseudo-orbit of $f$ is $\mathscr{M}_{\alpha}$-$\eps$-shadowed
by some point in $X$;

\item\label{6.5.6} for any $\eps>0$, there exists $\delta>0$
such that every $\delta$-asymptotic-average-pseudo-orbit is $\eps$-shadowed
on average by some point in $X$.
\end{enumerate}
\end{thm}

\begin{proof}
As a consequence of Theorem~\ref{Theorem 5.1}, Theorem~\ref{Theorem 5.3}, Theorem~\ref{Theorem 6.2},
Lemma~\ref{Lemma 6.3} and Lemma~\ref{Lemma 6.4}, we obtain that
$\eqref{6.5.6}\Longleftarrow\eqref{6.5.3}\Longrightarrow\eqref{6.5.1}$ and
$\eqref{6.5.1}\Longleftrightarrow \eqref{6.5.2}\Longleftrightarrow
 \eqref{6.5.4} \Longleftrightarrow \eqref{6.5.5}$.
Since every $\delta$-average-pseudo-orbit is also a
 $2\delta$-asymptotic-average-pseudo-orbit, it follows that $\eqref{6.5.6}\Longrightarrow
 \eqref{6.5.1}$. Therefore, we only
need to show that $\eqref{6.5.1}\Longrightarrow \eqref{6.5.3}$.

Fix an $\eps>0$ and an asymptotic average pseudo-orbit $\set{x_i}_{i=0}^\infty$.
The average shadowing property of $f$ implies that there
exists $\delta>0$ such that every $\delta$-average-pseudo-orbit is $\eps/2$-
shadowed on average by some point in $X$.

Clearly, $\set{x_i}_{i=0}^\infty$ is a $\gamma$-ergodic pseudo-orbit
for any $\gamma>0$, in particular for $\gamma=\delta/2$. Hence, by
Lemma~\ref{Lemma 6.1}, there exists a $\delta$-average pseudo-orbit
$\set{y_i}_{i=0}^\infty$ such that the set $\K=\set{i : x_i\neq
y_i}$ has density zero (i.e. $d(\K)=0$). By the choice of $\delta$,
there is a point $z\in X$ such that
$$
\limsup_{n\rightarrow\infty}\frac{1}{n}\sum_{i=0}^{n-1}d(f^{i}(z), y_{i})<\frac{\eps}{2}.
$$
But, similarly to the proof of \eqref{5.2}, we obtain that
$$
\limsup_{n\rightarrow\infty}\frac{1}{n}\sum_{i=0}^{n-1}d(f^{i}(z), x_{i})<\frac{\eps}{2}.
$$

We have just proved that $(X,f)$ has the weak average asymptotic
shadowing property, and so the proof is finished.
\end{proof}

\begin{cor}
If a dynamical system $(X,f)$ has the ergodic shadowing property, then it also has the average shadowing property.
\end{cor}
\begin{proof}
Simply, the family $\hat{\M}_1$ is contained in the
family $\M_\alpha$ for any $\alpha \in [0,1)$. Then, the result
follows directly by Theorem~\ref{Theorem 6.5}.
\end{proof}

Clearly, $\M_{0}$ and $\M^{1/2}$ are subsets of $\M_{1/2}$. Hence, Theorem~\ref{Theorem 6.5} immediately implies the following.
\begin{cor}\label{Cor_d}
If a dynamical system $(X,f)$ has the average shadowing property, then
it also has the $\underline{d}$-shadowing and $\overline{d}$-shadowing properties.
\end{cor}

\section{Dynamics on measure center and shadowing}

\subsection{The $\M_\alpha$-shadowing property on the measure center}
Motivated by \cite{KO11, KKO14}, this section is devoted to proving that the
$\mathscr{M}_{\alpha}$-shadowing property of a dynamical system restricted on
its measure center can ensure the same dynamical property of the entire system.
\begin{lem}\label{7.1}
Let $A\subset X$ be a closed invariant set containing the measure
center of a compact dynamical system $(X, f)$. Then, for every
asymptotic average pseudo-orbit $\left\{x_{i}\right\}_{i=0}^{\infty}$ of $f$,
there exists an asymptotic average pseudo-orbit $\left\{y_{i}\right\}_{i=0}^{\infty}
\subset A$ of $f|_{A}$ and a set $J\subset \Zp$ with density zero such that
$\lim_{i\notin J}d(x_{i}, y_{i})=0$.
\end{lem}
\begin{proof}
The proof is very technical but has a standard idea, hence we decided only to present how to derive it from existing results.
A careful reader should be able to write down a complete proof.

Corollary~2.4 in \cite{KKO14} shows that if a set $A$ contains the measure center of a compact
then for any $\eps > 0$ there exists $N>0$ such that
for any $x \in X$ and $n< N$ we have
$$
\frac{1}{n}|\set{ 0\leq i < n : d(f^i(x),A) < \eps}|>1-\eps.
$$
Then, $A$ satisfies the standing assumption of Lemmas 3.4--Lemma 3.9 in \cite{KO11} and hence we can repeat the first part of the proof in
\cite[Theorem 3.3]{KO11} on pp.42--43, obtaining an asymptotic average pseudo-orbit $\{y_{i}\}_{i=0}^{\infty}
\subset A$ of $f$ and a set $J'\subset \Zp$ with density zero such that
$\lim_{i\notin J'}d(x_{i}, y_{i})=0$.
\end{proof}

\begin{thm}\label{Theorem 7.2}
Let $A\subset X$ be a closed invariant set containing the measure
center of a compact dynamical system $(X, f)$. If $f|_{A}$ has the
$\mathscr{M}_{\alpha}$-shadowing property for some
$\alpha\in \left[0, 1\right)$ on $A$, then so does $f$ on $X$.
\end{thm}
\begin{proof}
For any asymptotic average pseudo-orbit $\left\{x_{i}\right\}_{i=0}^{\infty}$ of $f$
and any $\varepsilon>0$, it follows from Lemma~\ref{7.1} that
there exists an asymptotic average pseudo-orbit $\left\{y_{i}\right\}_{i=0}^{\infty}
\subset A$ of $f|_{A}$ and a set $J\subset \Zp$ with density zero such that
$\lim_{i\notin J}d(x_{i}, y_{i})=0$. This implies that there exists a set $J'\supset J$
with density zero such that for any $i\in \Zp\setminus J'$ we have $d(x_{i}, y_{i})
<\varepsilon/2$. The $\M_{\alpha}$-shadowing property of $f|_{A}$ implies that there exists
a point $z\in A$ such that
$$
\liminf_{n\rightarrow\infty}\frac{1}{n}\left|\Lambda_{n}(z, \left\{y_{i}\right\}_{i=0}^{\infty},
f|_{A}, \varepsilon/2)\right|>\alpha.
$$
But, for every
$$
i\in \Lambda(z, \left\{y_{i}\right\}_{i=0}^{\infty},
f, \eps/2) \setminus J',
$$
we have that
$$
d(f^{i}(z), x_{i})\leq d(f^{i}(x), y_{i})
+d(y_{i}, x_{i})<\eps,
$$
and hence
\begin{eqnarray*}
\liminf_{n\rightarrow\infty}\frac{1}{n}\left|\Lambda_{n}(z, \left\{x_{i}\right\}_{i=0}^{\infty},
f, \varepsilon)\right| &\geq& \liminf_{n\rightarrow\infty}\frac{1}{n}\left|\Lambda_{n}(z, \left\{y_{i}\right\}_{i=0}^{\infty},
f|_{A}, \varepsilon/2)\setminus J'\right|\\
&>&\alpha.
\end{eqnarray*}

It shows that every asymptotic average pseudo-orbit is
$\mathscr{M}_{\alpha}$-$\eps$-shadowed by a point in $X$, and then the result follows by
Theorem~\ref{Theorem 6.2}.
\end{proof}

As an immediate corollary, we obtain the following result which is \cite[Theorem 5.5]{KKO14}.
\begin{cor}\label{Corollary 7.3}
Let $A\subset X$ be a closed invariant set containing the measure
center of a compact dynamical system $(X, f)$. If $f|_{A}$ has the average shadowing
property on $A$, then so does $f$ on $X$.
\end{cor}
\begin{proof}
By Theorem~\ref{Theorem 7.2}, for any $\alpha \in [0,1)$,
if $f|_{A}$ has the $\M_\alpha$-shadowing property then $f$ has
the $\M_\alpha$-shadowing property. Then, the result follows by Theorem~\ref{Theorem 6.5}.
\end{proof}
We also obtain the following result, since the $\underline{d}$-shadowing
property is in fact the $\M_0$-shadowing property.

\begin{cor}\label{Corollary 7.4}
Let $A\subset X$ be a closed invariant set containing the measure
center of a compact dynamical system $(X, f)$. If $f|_{A}$ has
the $\underline{d}$-shadowing property on $A$,
then so does $f$ on $X$.
\end{cor}

\subsection{The (almost) specification properties and the measure center}
To proceed , we need the following important result, which is attributed to
Auslander and Ellis (see \cite[Theorem~8.7]{FurBook}).
\begin{lem}\label{lem:prox_min}
Let $(X,f)$ be a dynamical system and fix an $x\in X$. There exists $y\in M(f)$
such that $x,y$ are proximal, i.e. $\liminf_{n\to \infty} d(f^n(x),f^n(y))=0$.
\end{lem}

A point $x\in X$ is \emph{non-wandering} if there is a sequence $\set{x_n}_{n=1}^\infty$ and an increasing sequence of positive integers $\set{k_n}_{n=1}^\infty$ such that $\lim_{n\to \infty}x_n=x$ and $\lim_{n\to \infty}f^{k_n}(x_n)=x$.
As usual, we denote by $\Omega(f)$ the set of all non-wandering points.

\begin{thm}\label{thm:spec}
If $(X,f)$ has the specification property then the following conditions hold:
\begin{enumerate}
\item $M(f)$ is dense in $\Omega(f)$, in particular
$f(\Omega(f))=\Omega(f)$ and $\Omega(f)$ is the measure center for $f$;

\item $f|_{\Omega(f)}$ has the specification property.
\end{enumerate}
\end{thm}
\begin{proof}
Fix an $\eps>0$ and let $M$ be provided by the specification property.
Fix any $y_1,\ldots ,y_n\in \Omega(f)$ and any sequence of natural numbers $0\leq a_1\leq b_1 < a_2 \leq b_2 <\ldots \leq b_n$
such that for every $2\leq i \leq n$ we have $a_i-b_{i-1}\geq M$. Denote $p=b_n+M$.

Put $a_{jn+i}=pj+a_i$ and $b_{jn+i}=pj+b_i$ for $j=0,1,\ldots$ and $i=1,\ldots,n$.
Fix any $1\leq i\leq n$ and observe that since $y_i\in \Omega(f)$
there exists an increasing sequence $n_k$ and points $z_k$ such that $\lim_{k\to \infty} f^{n_k}(z_k)=y_i$.
In particular, passing to a subsequence if necessary, for every $j\geq 0$ we can find $k$ such that $n_k>jp$
and hence the following point is well defined:
$$
y_{jn+i}=\lim_{k\to \infty} f^{n_k-jp}(z_k).
$$
Clearly, $f^{jp}(y_{jn+i})=y_i$ and since $X$ is compact, we can apply the specification property also to infinite sequences of points and times. Therefore, there is $z$ such that for any $1\leq i \leq n$, any $j\geq 0$ and any $a_{jn+i}\leq k \leq b_{jn+i}$, we have $d(f^k(y_{jn+i}),f^k(z))\leq\varepsilon$.

Directly from definition, we obtain that
$$
d(f^k(y_{jn+i}),f^k(z))=d(f^{k-jp}(y_{i}),f^k(z)).
$$
By Lemma~\ref{lem:prox_min}, there exists $q'\in M(f)$ such that $q',z$ are proximal. In particular,
there is $j>0$ such that $d(f^k(z),f^k(q'))<\eps$ for every $jp\leq k \leq (j+1)p$. Denote $q=f^{jp}(q')$
and observe that for any $i=1,\ldots,n$ and any $a_i\leq s \leq b_i$, we have
\begin{eqnarray*}
d(f^s(q),f^s(y_i)) &\leq& d(f^{jp+s}(q'),f^{jp+s}(z))+d(f^{jp+s}(z),f^s(y_i))\\
&<&2\eps.
\end{eqnarray*}
This shows that in the definition of specification we can assume that $z\in M(f)$, provided that all the points $y_i \in \Omega(f)$.
This observation has several consequences.

First of all, any point in $\Omega(f)$ can be approximated by a point from $M(f)\subset \Omega(f)$, so indeed $\Omega(f)=\overline{M(f)}$. But $f(M(f))=M(f)$,
so we immediately obtain that $f(\Omega(f))=\Omega(f)$. Since for any choice of segments of orbits in $\Omega(f)$ the tracing point
in $X$ can also belong to $\Omega(f)$, we obtain that $f$ restricted to $\Omega(f)$ has the specification property.

Finally, since the measure center is contained in $\Omega(f)$ and every minimal set is a support of an invariant measure,
we obtain that $\Omega(f)$ is the measure center.
\end{proof}
\begin{thm}\label{mc:almost_s}
Let $(X,f)$ be a dynamical system and $Y$ be the measure center.
Then, $(X, f)$ has the almost specification property if and only if
$(Y,f|_Y)$ has the almost specification property.
\end{thm}
\begin{proof}
The sufficiency has been proved in \cite[Theorem 5.1]{KKO14}. To prove the necessity let us take
a mistake function $g$ for $f$ and let $k_g\colon (0,\infty)\longrightarrow \N$
be the second function from the almost specification property. For any $n$ and $\eps$, denote $G(n,\eps)=g(n,\eps/2)$.
Fix an $m\geq 1$, $\eps_1,\ldots,\eps_m > 0$, points $x_1, \ldots, x_m \in X$, and integers
$n_1 \geq k_{g}(\eps_1/2),\ldots,n_m \geq k_{g}(\eps_m/2)$. As usual, set $n_0=0$ and
$$
l_j=\sum_{s=0}^{j-1}n_s,\,\text{for }j=1, \ldots, m+1.
$$
Let $K=l_m$ and observe that by the almost specification property, for any $s$ there is a point $z_s$
such that for any $i=0, \ldots, s$ and any $j=1, \ldots, m$, we have
$$
f^{l_j+i K}(z_s)\in B_{n_j}(g;x_j,\eps_j/2).
$$
In other words, we repeat $s$-times periodically segments of the orbits of $x_1,\ldots, x_m$
and then use tracing provided by almost specification.
Without loss of generality, we may assume that there exists $z=\lim_{s\to \infty} z_s$.
By Lemma~\ref{lem:prox_min}, there exists a point $y\in M(f)$ such that $\liminf_{n\to \infty} d(f^n(y),f^n(z))=0$.
In particular, if we put $\eps=\min_{j} \eps_j/2$ then there is $r$ such that $d(f^t(z),f^t(y))<\eps$ for all integers $t\in [rK, (r+1)K]$.

Since $I(g;n_j,\eps_j/2)$ is finite, without loss of generality, we may assume that for any $j$ there is $A_j\in I(g;n_j,\eps_j/2)$
such that $f^{l_j+rK}(z_s)\in B_{A_j}(g;x_j,\eps_j/2)$ for any $s$ sufficiently large.
But, then, $d(f^{l_j+rK+i}(z), f^{i}(x_j))<\eps_j+\eps$ for every $i\in A_j$. Clearly,
$f^{rK}(z)\in M(f)$, hence in the definition of almost specification we may assume that
the tracing point comes from $M(f)$ (simply, we take as the mistake function $G$
and put $k_G(\gamma)=k_g(\gamma/2)$ for $\gamma>0$).

The result follows from the fact that every minimal set is a support of an
invariant measure, and so the measure center must contain the set $M(f)$.
\end{proof}

Examples presented in \cite{KKO14} show that it may happen that the map with almost specification is transitive
while the measure center is a single point. Hence it is not possible to repeat the statement of Theorem~\ref{thm:spec}
in the case of a dynamical system with the almost specification property. Still, we have the following property.

\begin{thm}
If $(X,f)$ has the almost specification property, then $\overline{M(f)}$ is the measure center of $f$.
\end{thm}
\begin{proof}
Let $A$ be the measure center for $(X,f)$. Clearly, $M(f)\subset A$.
Now,  let $U$ be any open set, such that $A\cap U\neq \emptyset$.
There is an open set $V$ and $\eps>0$ such that $B(\overline{V},\eps)\subset U$ and $V\cap A\neq \emptyset$.
Then, there is an $f$-invariant measure $\mu$
such that $\mu(V)>0$. Using ergodic decomposition \cite[p. 153]{W1982}, we get an
ergodic measure $\nu$ such that $\nu(V) > 0$. By the Pointwise Ergodic
Theorem \cite[Thmeorem 1.14]{W1982}, there is a point $x \in X$ such that $\underline{d}(N(x,V)) = \nu(V) > 0$.
In particular, there are $\lambda>0$ and $K>0$ such that $|N(x,V)\cap [0,n)|>\lambda n$ for all $n\geq K$.
We can increase $K$ to ensure also that $g(n,\eps)< (1-\lambda)n$ for all $n\geq K$.
By the technique employed in the proof of Theorem~\ref{mc:almost_s}, there is $z\in M(f)$ such that
$$
|\set{0\leq i \leq n : d(f^i(z),f^i(x))<\eps}|\geq (1-\lambda)n,
$$
and, hence, there is $i$ such that $f^i(x)\in V$ and $d(f^i(z),f^i(x))<\eps$,
showing that $f^i(z)\in U$. We have just proved that $M(f)\cap U\neq \emptyset$
and hence $A\subset \overline{M(f)}$. The proof is completed.
\end{proof}
\begin{cor}\label{Corollary 6.9}
If $(X,f)$ has the almost specification property, then it has the asymptotic average shadowing property.
\end{cor}
\begin{proof}
By Theorem~\ref{mc:almost_s}, $(Y, f|_Y)$ has the almost specification property. Clearly, $f|_Y$ is surjective,
hence we may apply \cite[Theorem~3.5]{KKO14} to obtain that $(Y, f|_Y)$ has the asymptotic average shadowing property. Now,
\cite[Theorem~5.1]{KKO14} implies that $(X,f)$ has the asymptotic average shadowing property, and so the proof is completed.
\end{proof}

\section{Sensitivity in systems with $\underline{d}$-shadowing and $\overline{d}$-shadowing  properties}\label{sec:sens}

In this section, we continue the research of \cite{DH2010} and \cite{Gu2007-1, KO2010, Niu2011, ODH2014},
proving some consequences of $\underline{d}$-shadowing and $\overline{d}$-shadowing
properties for the dynamics of the systems. Among other things, we show that in most cases these systems are
syndetically transitive and syndetically sensitive.

\subsection{Syndetic transitivity for
systems with the $d$-shadowing property}

This subsection shows that every system with dense minimal points having the $\overline{d}$-shadowing
property or the $\underline{d}$-shadowing property is totally syndetically transitive, which
extends Theorem 2.1 and Theorem 2.2 in \cite{Niu2011}.

The proof of the following Lemma is straightforward and is left to the reader.

\begin{lem}\label{Lemma 8.1}
If $(X, f)$ is a topologically transitive system with a dense set of
minimal points, then $(X, f)$ is totally syndetically transitive.
\end{lem}

\begin{thm}\label{Theorem 8.2}
Let $(X, f)$ be a dynamical system with a dense set of minimal points. If $(X,f)$ has
the $\overline{d}$-shadowing property or the $\underline{d}$-shadowing property, then it is totally syndetically transitive.
\end{thm}
\begin{proof}
First, we present a proof for the case of the $\overline{d}$-shadowing property.
Given any pair of nonempty open subsets $U, V\subset X$, there exist
$u\in U\cap M(f)$, $v\in V\cap M(f)$ and $\eps>0$,
such that $B(u, 2\eps)\subset U$ and $B(v, 2\eps)\subset V$. Points
$u, v$ are minimal, hence there exists $K>0$ such that for any $n\in \N$, we have $\left[n,
n+K\right]\cap N_{f}(u, B(u, \eps))\neq
\emptyset \neq \left[n, n+K\right]\cap N_{f}(v,
B(v, \eps))$. There exists $\delta>0$ such that for any $y, z\in X$,
$$
d(y, z)<\delta \quad \Longrightarrow \quad d(f^{n}(y), f^{n}(z))<\eps\quad \text{for all }0\leq n \leq K.
$$
Put $L_{1}=2$ and $L_{n}=L_{n-1}+n$ for $n\geq 2$ and then denote
$\mathscr{A}=\N \cap \bigcup_{n\in \mathbb{N}}\left[L_{2n},
L_{2n+1}\right)\cup \left[0,L_1\right)$ and $\mathscr{B}=\N \cap \bigcup_{n\in
\mathbb{N}}\left[L_{2n-1}, L_{2n}\right)$.
It is not difficult to
check that
$$
d\left(\mathbb{Z}^{+}\setminus \left\{L_{1}, L_{2},
\ldots\right\}\right)=1,
\quad \text{ and }\quad
d(\mathscr{A})=d(\mathscr{B})=\frac{1}{2}.
$$
Choose a sequence $\left\{x_{i}\right\}_{i=0}^{\infty}$ with
$$
x_{i}=\begin{cases}
f^{i}(u), & {\rm } \ i\in\mathscr{A}, \\
f^{i}(v), & {\rm } \ i\in\mathscr{B}.
\end{cases}
$$
Directly by definition, we obtain that
$\left\{x_{i}\right\}_{i=0}^{\infty}$ is a $\gamma$-ergodic
pseudo-orbit for any $\gamma>0$ and so by the $\overline{d}$-shadowing
property of $f$, there exists $x\in X$ such that $\overline{d}\left(\Lambda(x,
\left\{x_{i}\right\}_{i=0} ^{\infty}, \delta)\right)>1/2$.
But, then, both sets $\mathscr{A}\cap
\Lambda(x, \left\{x_{i}\right\}_{i=0}^{\infty}, \delta)$ and
$\mathscr{B}\cap \Lambda(x, \left\{x_{i}\right\}_{i=0}^{\infty}, \delta)$ are infinite. Therefore, there exist $s\in
\mathscr{A}\cap \Lambda(x, \left\{x_{i}\right\}_{i=0}^{\infty}, \delta)$ and $t\in \mathscr{B}\cap \Lambda(x,
\left\{x_{i}\right\}_{i=0}^{\infty}, \delta)$ such that
$K < t-s$. Clearly, $d(f^{s}(x), f^{s}(u))=d(f^{s}(x), x_{s})<\delta$ and $d(f^{t}(x), f^{t}(v))
=d(f^{t}(x), x_{t})<\delta$, and points $u,v$ are minimal hence there exist $0\leq s', t'\leq K$
such that
$f^{s+s'}(u)\in B(u, \eps)$ and $f^{t+t'}(v)\in
B(v, \eps)$ and also $d(f^{s+s'}(u),f^{s+s'}(x))<\eps$ and $d(f^{t+t'}(v),f^{t+t'}(x))<\eps$.
In particular, $(t+t')-(s+s')>0$ and $f^{s+s'}(x)\in B(u,2\eps)\subset U$, $f^{t+t'}(x)\in B(v,2\eps)\subset V$.
This proves that $(X,f)$ is transitive, hence syndetically transitive by Lemma \ref{Lemma 8.1}.

By Corollary~\ref{dshad_higher}, the dynamical system $(X,f^n)$ has $\underline{d}$-shadowing for
every $n=1,2,\ldots$, which completes the proof of the case of the $\overline{d}$-shadowing property.

If $(X,f)$ has the $\underline{d}$-shadowing property then the same proof works,
with the only modification of the definitions
$l_{1}=L_{1}=2$, $l_{n}=2^{l_{1}+\cdots +l_{n-1}}$ and
$L_{n}=l_{1}+\cdots +l_{n}$ for $n\geq 2$.
\end{proof}

Since $\M_0$ contains $\M_\alpha$ for any $\alpha \in [0,1)$, we obtain the following.

\begin{cor}\label{Corollary 8.3}
If a dynamical system $(X, f)$ with a dense set of minimal points has
the $\mathscr{M}^{\alpha}$-shadowing property for some $\alpha\in \left[1/2, 1\right)$
or the $\mathscr{M}_\alpha$-shadowing for some $\alpha\in [0,1)$, then
it is totally syndetically transitive.
\end{cor}
\begin{rem}
Example \ref{Example 4.4} shows that for a system with dense minimal points, the $\mathscr{M}^{\alpha}$-shadowing
property ($\alpha\in \left[0, 1/2\right)$) is not sufficient for having the transitivity.
\end{rem}

\begin{cor}\label{Corollary 8.4}
If a dynamical system $(X, f)$ with a dense set of minimal points has the $\mathscr{M}_\alpha$-shadowing
for some $\alpha\in\left[1/2, 1\right)$, then it is weakly mixing.
\end{cor}
\begin{proof}
Applying \eqref{**}, it is clear that for any $A, B\subset \Zp$ and any $n\in \mathbb{N}$,
\begin{eqnarray*}
\frac{|A\cap B \cap\{0, 1, \ldots, n-1\}|}{n}&=&\frac{|A\cap\{0, 1, \ldots, n-1\}|}{n}
+\frac{|B \cap\{0, 1, \ldots, n-1\}|}{n}\\
&-&\frac{|\left(A\cup B\right) \cap\{0, 1, \ldots, n-1\}|}{n}.
\end{eqnarray*}
This implies that for any $A, B\in \M_\alpha$, we have $A\cap B\in \M_{0}$.
In other words, $(X\times X, f\times f)$ has the $\M_{0}$-shadowing property.
By \cite[Lemma~2.8]{AG2001}, the minimal points are dense for
$(X\times X, f\times f)$ and so Theorem~\ref{Theorem 8.2}
implies its transitivity, consequently $(X,f)$ is weakly mixing.
\end{proof}

\subsection{The $\underline{d}$-shadowing property and equicontinuity}
We now prove that the $\overline{d}$-shadowing for surjective systems implies
non-equicontinuity.
First, we need the following auxiliary Lemma.

\begin{lem}\label{Lemma 8.6}
If a dynamical system $(X, f)$ is equicontinuous, has the $\overline{d}$-shadowing property or
the $\underline{d}$-shadowing property, and $f$ is surjective, then
$(X,f)$ is totally transitive.
\end{lem}
\begin{proof}
We start the proof for the case of the $\overline{d}$-shadowing property.
Given any pair of nonempty open subsets $U, V\subset X$, pick $u\in
U$, $v\in V$ and $\gamma>0$ such that $B(u, \gamma)\subset U$ and
$B(v, \gamma)\subset V$. As $f$ is equicontinuous, it follows that
there exists $\varepsilon>0$ such that for any $y, z\in X$,
\begin{equation}\label{7.5}
d(y, z)<\varepsilon \quad \Longrightarrow \quad d(f^{n}(y), f^{n}(z))<\frac{\gamma}{2} \quad\text{ for all }n=0,1,2,\ldots.
\end{equation}
Let $L_{1}=2$ and put $L_{n}=L_{n-1}+n$
for $n\geq 2$. For each $n\in \mathbb{N}$, set
$$
\mathscr{A}_{n}=\left[L_{2n}, L_{2n+1}\right), \text{ and } \mathscr{B}_{n}=\left[L_{2n-1}, L_{2n}\right),
$$
and then denote
$\mathscr{A}=\bigcup_{n\in \mathbb{N}}\mathscr{A}_{n}$ and $\mathscr{B}=\bigcup_{n\in \mathbb{N}}\mathscr{B}_{n}$.
Since $f$ is surjective, we can find a sequence $\set{v_{-j}}_{j=0}^\infty$ such that $v_{-j+1}=f(v_{-j})$ for all $j\in \N$ and $v_{0}=v$.
Define a sequence $\left\{x_{i}\right\}_{i=0}^{\infty}$ by
\[
x_{i}=\left\{\begin{array}{cc}
f^{i}(u), & {\rm } \ i\in\left[0, L_{1}\right), \\\
f^{i-L_{2n}}(u), & {\rm } \ i\in\mathscr{A}_{n} \text{ for some } n\in \mathbb{N}, \\\
v_{i-L_{2n}}, & {\rm } \ i\in\mathscr{B}_{n} \text{ for some } n\in \mathbb{N}.
\end{array}
\right.
\]
It can be verified that for any $\delta>0$, the sequence $\left\{x_{i}\right\}_{i=0}^{\infty}$ is a
$\delta$-ergodic pseudo-orbit. Then, we can apply the
$\overline{d}$-shadowing property of $f$ to obtain
$x\in X$ such that $\overline{d}\left(\Lambda(x,
\left\{x_{i}\right\}_{i=0} ^{\infty}, f, \varepsilon)\right)>1/2$.
But $d(\mathscr{A})=d(\mathscr{B})=1/2$, so each of them intersects $\Lambda(x,
\left\{x_{i}\right\}_{i=0} ^{\infty}, f, \varepsilon)$ infinitely many times.
Hence, there are $i, j, n_{i}, n_{j}\in \mathbb{N}$ such that $n_i<n_j$
and $i\in \mathscr{A}_{n_{i}}\cap \Lambda(x,
\left\{x_{i}\right\}_{i=0}^{\infty}, f, \varepsilon)$, $j\in
\mathscr{B}_{n_{j}}\cap \Lambda(x, \left\{x_{i}\right\}_{i=0}^{\infty}, f,
\varepsilon)$. This implies that
$$
d(f^{i}(x), f^{i-L_{2n_{i}}}(u))<\eps\text{ and } d(f^{j}(x), v_{j-L_{2n_{j}}})<\eps,
$$
which, by (\ref{7.5}), implies that
\[
d(v, f^{L_{2n_{j}}}(x))=
d\left(f^{L_{2n_{j}}-j}(v_{j-L_{2n_{j}}}), f^{L_{2n_{j}}-j}(f^{j}(x))\right)<\frac{\gamma}{2},
\]
and
\[
d\left(f^{L_{2n_{j}}-L_{2n_{i}}}(u), f^{L_{2n_{j}}}(x)\right)=d\left(f^{L_{2n_{j}}-i}
(f^{i-L_{2n_{i}}}(u)), f^{L_{2n_{j}}-i}
(f^{i}(x))\right)<\frac{\gamma}{2}.
\]
Then,
\[
d\left(v, f^{L_{2n_{j}}-L_{2n_{i}}}(u)\right)
\leq d(v, f^{L_{2n_{j}}}(x))+ d\left(f^{L_{2n_{j}}-L_{2n_{i}}}(u), f^{L_{2n_{j}}}(x)\right)<\gamma.
\]
This implies that $f^{L_{2n_{j}}-L_{2n_{i}}}(U)\cap V\neq \emptyset$. This shows that $(X,f)$ is transitive.
But all the assumptions of the Theorem are satisfied also by $(X,f^n)$ (see Corollary~\ref{dshad_higher}), hence the proof is finished.

If $(X,f)$ has the $\underline{d}$-shadowing property then the same proof works, with the only modification of the definition of $L_n$, by putting
$l_{1}=L_{1}=2$, $l_{n}=2^{l_{1}+\cdots +l_{n-1}}$ and
$L_{n}=l_{1}+\cdots +l_{n}$ for $n\geq 2$.
\end{proof}


\begin{thm}\label{Theorem 8.7}
If $(X, f)$ is a
nontrivial equicontinuous dynamical system and $f$ is surjective, then $f$ does not have
the $\underline{d}$-shadowing property.
\end{thm}
\begin{proof}
Suppose that $f$ has the $\underline{d}$-shadowing property.
Then, by Lemma~\ref{Lemma 8.6}, it is transitive and consequently it follows that
$f$ is a minimal homeomorphism (e.g. see \cite[Theorem 4]{AY1980}).
Combining this with~\cite[Theorem 2.8]{DH2010}, it follows that
$(X,f)$ is weakly mixing. Clearly, each weakly mixing system is sensitive,
except the case when $|X|=1$. Both situations are impossible by our assumptions,
hence $(X,f)$ cannot have the $\underline{d}$-shadowing property and so the proof
is completed.
\end{proof}
\begin{rem}
We know, based on the previous results, that
$$
\text{AASP}\Longrightarrow \text{ASP} \Longrightarrow \underline{d}\text{-shadowing},
$$
hence Theorem~\ref{Theorem 8.7} implies that a dynamical system
satisfying the conditions of \cite[Theorem 3.1]{Gu2007}
(nontrivial equicontinuous surjective dynamical system with the AASP)
is trivial.
\end{rem}

\subsection{The $\overline{d}$-shadowing property and equicontinuity}
It is well known that there are minimal dynamical systems with shadowing
(they are exactly the odometers, e.g. see \cite{Mai}).
It is not the case for ASP for which such situation is impossible, since it was
proved in \cite{KKO14} that every dynamical system with ASP and a fully
supported measure is weakly mixing. However, since it was recently proved
that every topological $K$-system has $\underline{d}$-shadowing
(see \cite[Theorem~11]{ODH2014}), ASP can also be found within minimal
dynamical systems. While we still do not have an answer whether
the $\overline{d}$-shadowing property can exist in a nontrivial minimal system,
we can prove that, similarly to the $\underline{d}$-shadowing property,
such a system should at least be weakly mixing.

\begin{lem}\label{Lemma 8.9}
If $(X, f)$ factors onto a non-trivial equicontinuous minimal system $(Y, g)$, then it does
not have the $\overline{d}$-shadowing property.
\end{lem}

\begin{proof}
First, we show that $(Y, g)$ does not have the $\overline{d}$-shadowing property.
Since every equicontinuous minimal dynamical system is conjugated to
an isometry (see \cite{AkinTC}), and the $\overline{d}$-shadowing
property is preserved by the topological conjugation (see Proposition~\ref{Proposition 4.1}),
without loss of generality we may assume that
$(Y, g)$ is endowed with a metric $d$ such that $g$ is an isometry.
Choose two distinct points $y, y'\in Y$, and denote
$\xi=d(y, y')$ and $\delta=\xi/4$. It follows from the minimality
of $(Y, g)$ that $N_{g}(y, B(y', \delta))$ is syndetic,
that is, there exists $K\geq 4$ such that
$(n, n+K]\cap N_{g}(y, B(y', \delta))\neq \emptyset$ for all $n\geq 0$.

Put $m_{0}=0$ and then take $m_{n}=2n(n+1)K$ for all integers $n\geq 1$.
Finally, define inductively an increasing sequence $\set{{M}_{n}}
_{n=0}^{\infty}$ satisfying the following conditions:
\begin{enumerate}[(i)]
\item ${M}_{0}=0$;
\item ${M}_{n}\in \left[m_{n}, m_{n}+K\right]$ for all $n=1, 2, \ldots$;
\item for every $n\geq 0$ there is
$$
j\in \left(m_{n}-{M}_{n-1}, m_{n}-{M}_{n-1}
+K\right]\cap N_{g}(y, B(y', \delta)).
$$
such that $M_{n}=M_{n-1}+j$.
\end{enumerate}
The above construction immediately implies that $g^{{M}_{n}-{M}_{n-1}}
(y)\in B(y', \delta)$.

Denote $J=\left\{{M}_{1},
{M}_{2}, \ldots\right\}$ and put
$\mathscr{A}_{n}=\left[{M}_{n}, {M}_{n+1}\right)$ for $n=0,1,\ldots$.
Note that $m_{n+2}-m_{n+1}-K\geq m_{n+1}+K-m_{n}$ and so
\begin{equation}\label{7.2}
\left|\mathscr{A}_{n+1}\right|\geq |\mathscr{A}_{n}| \text{ for all } n=0, 1, 2, \ldots.
\end{equation}

%
%

Define a sequence $\left\{x_{i}\right\}_{i=0}^{\infty}$ by the formula
$$
x_{i}=
\begin{cases}
g^{i}(y), & \text{ if } i\in\mathscr{A}_{0}, \\\
g^{i-{M}_{n}}(y), & \text{ if } i\in\mathscr{A}_{n} \text{ for some } n\in \N.
\end{cases}
$$
Clearly, $d\left(J\right)=0$, so the sequence
$\set{x_{i}}_{i=0}^{\infty}$ is an $\varepsilon$-ergodic pseudo-orbit of $g$.

We claim that
\begin{equation}\label{7.3}
\limsup_{n\rightarrow\infty}
\frac{1}{n}\left|\Lambda_{n}(z, \set{x_{i}}_{i=0}^{\infty}, g, \delta)\right|
\leq 1/2
\end{equation}
for all $z\in Y$.

Fix an $r\in \Lambda(z, \set{x_{i}}_{i=0}^{\infty}, g, \delta)$
and let $n$ be such that $r\in \mathscr{A}_{n}$. Then, $x_{r}=g^{r-{M}_{n}}(y)$,
which gives
$$
d\left(g^{r}(z), g^{r-{M}_{n}}(y)\right)<\delta.
$$
Therefore, since $g$ is an isometry and $r< M_{n+1}$, we also have
$$
d\left(g^{{M}_{n+1}}(z), g^{{M}_{n+1}-{M}_{n}}(y)\right)
<\delta.
$$
Combining this with $d\left(y', g^{{M}_{n+1}-{M}_{n}}(y)\right)<\delta$,
we obtain that $d\left(g^{{M}_{n+1}}(z), y'\right)<2\delta$ and hence
$$
d\left(g^{{M}_{n+1}}(z), y\right)\geq d(y, y')-d\left(g^{{M}_{n+1}}(z), y'\right)
\geq 2\delta.
$$
Consequently, for every $j\in \mathscr{A}_{n+1}$ we have
$$
d(g^{j}(z), x_{j})=d\left(g^{j-{M}_{n+1}}(g^{{M}_{n+1}}(z)),
g^{j-{M}_{n+1}}(y)\right)\geq 2\delta.
$$
This means that
\begin{equation}\label{7.4}
\mathscr{A}_{n}\cap \Lambda(z, \set{x_{i}}_{i=0}^{\infty}, g, \delta)
\neq \emptyset\quad \Longrightarrow \quad
\mathscr{A}_{n+1}\subset \Lambda^{c}(z, \set{x_{i}}_{i=0}^{\infty}, g, \delta).
\end{equation}
Thus, for any $n\in \mathbb{N}$ and any $j\in \mathscr{A}_{n+2}$, we have
$$
\mathscr{A}_{n}\subset \Lambda_{j}^{c}(z, \set{x_{i}}_{i=0}^{\infty}, g, \delta),
$$
provided that $\Lambda(z, \set{x_{i}}_{i=0}^{\infty}, g, \delta)\cap \mathscr{A}_{n+1}\neq \emptyset$, but
$$
\mathscr{A}_{n+1}\subset \Lambda_{j}^{c}(z, \set{x_{i}}_{i=0}^{\infty}, g, \delta),
$$
otherwise. Denote $D_{n}=\left\{0\leq i\leq n: \Lambda(z, \set{x_{i}}_{i=0}^{\infty}, g, \delta)
\cap \mathscr{A}_{i}\neq \emptyset\right\}$. Clearly, applying \eqref{7.4} implies that for any
$i\in D_{n}+1$, $\mathscr{A}_{i}\subset \Lambda^{c}_{j}(z, \set{x_{i}}_{i=0}^{\infty}, g, \delta)$.
Combining this with \eqref{7.2}, it follows that
\begin{eqnarray*}
\left|\Lambda_{j}^{c}(z, \{x_{i}\}_{i=0}^{\infty}, g, \delta)\right|& \geq & \sum_{i\in D_{n}+1}|\mathscr{A}_{i}|\geq \sum_{i\in D_{n}+1}|\mathscr{A}_{i-1}|\\
& = & \sum_{i\in D_{n}}|\mathscr{A}_{i}|\geq |\Lambda_{M_{n+1}}(z, \{x_{i}\}_{i=0}^{\infty}, g, \delta)|,
\end{eqnarray*}
and hence
\begin{eqnarray*}
|\Lambda^c_{j}(z, \{x_{i}\}_{i=0}^{\infty}, g, \delta)|&\geq & \frac{1}{2}
\left(|\Lambda_{M_{n+1}}(z, \{x_{i}\}_{i=0}^{\infty}, g, \delta)|
+|\Lambda_{M_{n+1}}^{c}(z, \{x_{i}\}_{i=0}^{\infty}, g, \delta)|\right)\\
& =& \frac{M_{n+1}}{2}.
\end{eqnarray*}
This implies that
\begin{eqnarray*}
\frac{1}{j}\left|\Lambda_{j}^{c}(z, \set{x_i}_{i=0}^\infty, g, \delta)\right|&\geq&
\frac{M_{n+1}}{2M_{n+3}}\geq\frac{m_{n+1}}{2(m_{n+3}+K)}\\
& =& \frac{2K(n+1)(n+2)}{4K(n+3)(n+4)+2K}\longrightarrow \frac{1}{2}.
\end{eqnarray*}
Therefore, we obtain that
$$
\limsup_{n\rightarrow\infty}\frac{1}{n}\left|\Lambda_{n}(z, \set{x_{i}}_{i=0}^{\infty}, g, \delta)\right|
=1-\liminf_{n\rightarrow\infty}\frac{1}{n}\left|\Lambda_{n}^{c}(z, \set{x_{i}}_{i=0}^{\infty}, g, \delta)\right|\leq\frac{1}{2}.
$$
Indeed, the claim holds, and hence $(Y,g)$ does not have the $\overline{d}$-shadowing property.

Now, suppose that $f$ has the $\overline{d}$-shadowing property.
Fix any distinct $y, y'\in Y$ (by the assumptions $|Y|>1$), and let
$\xi, \delta$ and $\{x_{i}\}_{i=0}^{\infty}$, etc. be defined as at the start of the proof. Take any $\beta>0$ such that for any $u, v\in X$,
$$
d(u, v)<\beta\quad \Longrightarrow \quad d(\pi(u), \pi(v))<\delta.
$$
Take a $\nu\in X$ such that $\pi(\nu)=y$ and define a sequence $\left\{\nu_{i}\right\}_{i=0}^{\infty}$ by
$$
\nu_{i}=
\begin{cases}
f^{i}(\nu), & {\rm if } \ i\in\mathscr{A}_{0}, \\\
f^{i-M_{n}}(\nu), & {\rm if} \ i\in\mathscr{A}_{n} \text{ for some } n\in \mathbb{N}.
\end{cases}
$$

Recall that $d(J)=0$. Hence, for any $\gamma>0$ the sequence $\{\nu_{i}\}_{i=0}^{\infty}$ is
a $\gamma$-ergodic pseudo-orbit of $f$. Directly from definition, we also obtain that $\{\pi(\nu_{i})\}_{i=0}^{\infty}=\{x_{i}\}_{i=0}^{\infty}$.
Taking sufficiently small $\gamma$, the $\overline{d}$-shadowing property of $f$ implies that there exists $\widehat{z}\in X$
such that $\limsup_{n\rightarrow\infty}\frac{1}{n}\left|\Lambda_{n}(\widehat{z},
\left\{\nu_{i}\right\}_{i=0}^{\infty}, f, \beta)\right|>1/2$. Combining this with the fact
that for any $k\in \Lambda(\widehat{z},
\left\{\nu_{i}\right\}_{i=0}^{\infty}, f, \beta)$ we have
$$
d(g^{k}(\pi(\widehat{z})), \pi(\nu_{k}))
=d(\pi(f^{k}(\widehat{z})), \pi(\nu_{k}))<\delta,$$
it implies that
\begin{eqnarray*}
\limsup_{n\rightarrow\infty}\frac{1}{n}\left|\Lambda_{n}(\pi(\widehat{z}), \left\{x_{i}\right\}
_{i=0}^{\infty}, g, \delta)\right|&=&\limsup_{n\rightarrow\infty}\frac{1}{n}\left|\Lambda_{n}(\pi(\widehat{z}), \left\{\pi(\nu_{i})\right\}
_{i=0}^{\infty}, g, \delta)\right|\\
&\geq& \limsup_{n\rightarrow\infty}\frac{1}{n}\left|\Lambda_{n}(\widehat{z},
\left\{\nu_{i}\right\}_{i=0}^{\infty}, f, \beta)\right|>\frac{1}{2},
\end{eqnarray*}
which contradicts \eqref{7.3}. Hence, $(X, f)$ does not have the $\overline{d}$-shadowing
property.
\end{proof}

\begin{thm}\label{Theorem 8.10}
Every minimal system with the $\overline{d}$-shadowing property is weakly mixing.
\end{thm}
\begin{proof}
It is known (e.g. see \cite[Theorem 2.3]{GW05}) that a minimal dynamical system is weakly mixing
if and only if its maximal equicontinuous factor is trivial. Then, it suffices to apply
Lemma \ref{Lemma 8.9} and the result follows.
\end{proof}

Now, we have enough tools to repeat the proof of Theorem~\ref{Theorem 8.7} to obtain the following.
\begin{cor}\label{Corollary 8.11}
Let $f: X\longrightarrow X$ be a surjection. If $(X, f)$ is a
nontrivial equicontinuous dynamical system, then $f$ does not have
the $\overline{d}$-shadowing property.
\end{cor}

\begin{rem}
Every rotation $f: \mathbb{S}^{1}\longrightarrow \mathbb{S}^{1}$ on the unit circle
$\mathbb{S}^{1}$ is chain transitive, hence chain mixing (see \cite[Example 13]{RW2008}).
If it is a rational rotation, then there exists some $n\in \mathbb{N}$
such that $f^{n}$ is the identity map, and then by Corollary~\ref{dshad_higher} and \ref{Corollary 8.11} we see that
$f$ does not have the $\overline{d}$-shadowing property. If it is an irrational rotation,
Theorem~\ref{Theorem 8.10} indicates that $f$ does not have the $\overline{d}$-shadowing property.
This shows that the converse of Corollary~\ref{Corollary XX} is not true.
\end{rem}

By \cite[Proposition 1]{S2007}, any sensitive map with dense minimal points is syndetically sensitive, and by
\cite[Theorem 1]{S2007} any syndetically transitive nonminimal dynamical system is syndetically sensitive.
This leads to the following.

\begin{cor}\label{Corollary 8.12}
If a dynamical system $(X, f)$ has the $\underline{d}$-shadowing
property or the $\overline{d}$-shadowing property, and the minimal points of $f$
are dense in $X$, and if $X$ has at least two elements, then for every $n\in \mathbb{N}$,
dynamical system $(X,f^{n})$ is syndetically sensitive.
\end{cor}
\begin{proof}
By Theorem \ref{Theorem 8.2}, $(X,f)$ is totally syndetically transitive,
and when minimal, it is weakly mixing since it has a trivial maximal equicontinuous factor
(e.g. see \cite[Theorem~2.8]{DH2010} and Theorem~\ref{Theorem 8.10}). This together
with \cite[Proposition 1]{S2007} implies that $(X, f^{n})$ is syndetically sensitive
for any $n\in \mathbb{N}$. In the later case, $(X, f^{n})$ is syndetically sensitive
by application of \cite[Theorem 1]{S2007}.
\end{proof}

\section*{Acknowledgements}

The research of X. Wu was supported by YBXSZC20131046, the Scientific Research
Fund of the Sichuan Provincial Education Department (No. 14ZB0007).

Research of P.\ Oprocha was supported by Narodowe Centrum Nauki
(National Science Center) in Poland, grant no.\ DEC-2011/03/B/ST1/00790.
Some of the results presented in this paper were obtained when P. Oprocha was visiting
Max Planck Institute for Mathematics in Bonn in June 2014, taking part in the activity ``Dynamics and Numbers".
Hospitality and simulating, friendly atmosphere at MPIM is gratefully acknowledged.

The research of G. Chen was supported by
the Hong Kong Research Grants Council under GRF grant CityU 1109/12.



\end{document}